\documentclass[12pt]{amsart}
\usepackage[latin1]{inputenc}
\usepackage{mathptmx}
\usepackage{amscd}
\usepackage{amssymb}
\newtheorem{theorem}{Theorem}
\newtheorem{lemma}[theorem]{Lemma}
\newtheorem{corollary}[theorem]{Corollary}
\newtheorem{proposition}[theorem]{Proposition}

\theoremstyle{definition}
\newtheorem{remark}[theorem]{Remark}
\newtheorem{definition}[theorem]{Definition}

\newtheorem{example}[theorem]{Example}

\let\phi=\varphi

\def\Ann{\operatorname{Ann}}
\def\Ass{\operatorname{Ass}}
\def\Min{\operatorname{Min}}
\def\Spec{\operatorname{Spec}}

\let\oldbigwedge\bigwedge
\def\BIGwedge{{\textstyle\oldbigwedge}}
\def\medwedge{{\scriptstyle\oldbigwedge}}
\def\bigwedge{\mathchoice{\BIGwedge}{\BIGwedge}{\medwedge}{}}

\DeclareMathOperator{\Nil}{Nil}

\DeclareMathOperator{\FId}{FId}
\DeclareMathOperator{\Id}{Id}

\DeclareMathOperator{\Sub}{Sub}
\DeclareMathOperator{\zd}{zd}

\DeclareMathOperator{\supp}{supp}

\let\epsilon=\varepsilon

\begin{document}
\title{On the Content of Polynomials Over Semirings and Its Applications}

\author{Peyman Nasehpour}

\address{Peyman Nasehpour, Department of Engineering Science, Faculty of Engineering, University of Tehran, Tehran, Iran}
\email{nasehpour@gmail.com}

\begin{abstract}
In this paper, we prove that Dedekind-Mertens lemma holds only for those semimodules whose subsemimodules are subtractive. We introduce Gaussian semirings and prove that bounded distributive lattices are Gaussian semirings. Then we introduce weak Gaussian semirings and prove that a semiring is weak Gaussian if and only if each prime ideal of this semiring is subtractive. We also define content semialgebras as a generalization of polynomial semirings and content algebras and show that in content extensions for semirings, minimal primes extend to minimal primes and discuss zero-divisors of a content semialgebra over a semiring who has Property (A) or whose set of zero-divisors is a finite union of prime ideals. We also discuss formal power series semirings and show that under suitable conditions, they are good examples of weak content semialgebras.
\end{abstract}

\maketitle

\begin{center} Dedicated to Professor Winfried Bruns  \end{center}

\section*{0. Introduction}

Semirings not only have significant applications in different fields such as automata theory in theoretical computer science, (combinatorial) optimization theory, and generalized fuzzy computation, but are fairly interesting generalizations of two broadly studied algebraic structures, i.e. rings and bounded distributive lattices. Especially polynomials and formal power series over semirings have important role in applied mathematics\footnote{2010 Mathematics Subject Classification: 16Y60, 13B25, 13F25, 06D75.
Keywords: Semiring, Bounded distributive lattice, Semimodule, Semialgebra, Semiring polynomials, Semimodule polynomials, Content semimodule, Content semialgebra, Weak content semialgebra, Dedekind-Mertens content formula, Gauss' lemma, Few zero-divisors, McCoy's property, Minimal prime, Property (A), Primal semiring, Gaussian semiring, Weak Gaussian semiring}.

One of the interesting and helpful concepts for studying polynomial rings is the concept of the content of a polynomial. In fact, there are important connections between the contents of two polynomials over a ring and the content of their multiplication.

More precisely, let for the moment, $(R,+,\cdot,0,1)$ be a commutative ring with identity, $X$ an indeterminate over the ring $R$ and define the content of a polynomial $f \in R[X]$, denoted by $c(f)$, to be the $R$-ideal generated by the coefficients of $f$. A celebrated theorem in the multiplicative ideal theory of commutative rings, known as Dedekind-Mertens content formula that is a generalization of Gauss' lemma on primitive polynomials, states that for all $f,g \in R[X]$, there exists a nonnegative integer $m \leq \deg(g)$ such that $c(f)^m c(fg) = c(f)^{m+1} c(g)$ (\cite{AG}).

Since the concept of the content of a polynomial and Dedekind-Mertens content formula have some interesting applications in commutative algebra (cf. \cite{AG}, \cite{AK}, \cite{BG}, \cite{HH}, \cite{Na}, \cite{No}, \cite{OR}, \cite{R} and \cite{T}) and much of the theory of rings continues to make sense when applied to arbitrary semirings (cf. \cite{G} or \cite{HW}), the question may arise if similar applications can be found for polynomials over semirings. The main purpose of the present paper is to investigate the content of polynomials over semirings and to show their applications.

First we explain what we mean by a semiring. More on semirings can be found in the books \cite{G} and \cite{HW} for example. In this paper, by a semiring, we understand an algebraic structure, consisting of a nonempty set $S$ with two operations of addition and multiplication such that the following conditions are satisfied:

\begin{enumerate}
\item $(S,+)$ is a commutative monoid with identity element $0$;
\item $(S,\cdot)$ is a commutative monoid with identity element $1 \not= 0$;
\item Multiplication distributes over addition, i.e. $a\cdot (b+c) = a \cdot b + a \cdot c$ for all $a,b,c \in S$;
\item The element $0$ is the absorbing element of the multiplication, i.e. $s \cdot 0=0$ for all $s\in S$.
\end{enumerate}

A nonempty subset $I$ of a semiring $S$ is said to be an ideal of $S$, if $a+b \in I$ for all $a,b \in I$ and $sa \in I$ for all $s \in S$ and $a \in I$. Similar to the ideal theory of commutative rings, it is easy to see that if $a_1, a_2, \ldots , a_n \in S$, then the finitely generated ideal $(a_1, a_2, \ldots , a_n)$ of $S$ is the set of all linear combinations of the elements $a_1, a_2, \ldots , a_n$, i.e. $$(a_1, a_2, \ldots , a_n) = \{ s_1 a_1 + s_2 a_2 + \cdots + s_n a_n : s_1, s_2, \ldots, s_n \in S \}.$$

A nonempty subset $P$ of a semiring $S$ is said to be a prime ideal of $S$, if $P \not= S$ is an ideal of $S$ such that $ab \in P$ implies either $a\in P$ or $b\in P$ for all $a,b \in S$.

An ideal $I$ of a semiring $S$ is said to be subtractive, if $a+b \in I$ and $a \in I$ implies $b \in I$ for all $a,b \in S$. We say that a semiring $S$ is subtractive if every ideal of the semiring $S$ is subtractive.

Now let $(M,+,0)$ be a commutative additive monoid. The monoid $M$ is said to be an $S$-semimodule if there is a function, called scalar product, $\lambda: S \times M \longrightarrow M$, defined by $\lambda (s,m)= sm$ such that the following conditions are satisfied:

\begin{enumerate}
\item $s(m+n) = sm+sn$ for all $s\in S$ and $m,n \in M$;
\item $(s+t)m = sm+tm$ and $(st)m = s(tm)$ for all $s,t\in S$ and $m\in M$;
\item $s\cdot 0=0$ for all $s\in S$ and $0 \cdot m=0$ and $1 \cdot m=m$ for all $m\in M$.
\end{enumerate}

A nonempty subset $N$ of an $S$-semimodule $M$ is said to be an $S$-subsemimodule of $M$, if $m+n \in N$ for all $m,n \in N$ and $sn \in N$ for all $s \in S$ and $n \in N$. Similar to module theory over commutative rings, it is, then, easy to see that if $m_1, m_2, \ldots , m_n \in M $, then the finitely generated $S$-subsemimodule $(m_1, m_2, \ldots , m_n)$ of $M$ is the set of all linear combinations of the elements $m_1, m_2, \ldots, m_n$, i.e. $$(m_1, m_2, \ldots , m_n) = \{ s_1 m_1 + s_2 m_2 + \cdots + s_n m_n : s_1, s_2, \ldots, s_n \in S \}.$$

Note that if $I$ is an ideal of the semiring $S$ and $N$ is an $S$-subsemimodule of $M$, the set $IN = \{s_1 m_1 + s_2 m_2 + \cdots + s_n m_n: s_i \in S, m_i \in M, n \in \mathbb N\}$ is also an $S$-subsemimodule of $M$.

An $S$-subsemimodule $N$ of an $S$-semimodule $M$ is said to be subtractive if $m, m+n \in N$ for all $m,n \in M$ implies $n \in N$. We say that an $S$-semimodule $M$ is subtractive if every $S$-subsemimodule of $M$ is subtractive. One can easily check that if every $2$-generated $S$-subsemimodule of the $S$-semimodule $M$ is subtractive, then $M$ is a subtractive semimodule. Subtractive semimodules and semirings play a central role in this paper.

Though the above definitions are rather standard, we brought them here to close the window of possible ambiguities.

In the first section of the present paper, we prove that if $S$ is a semiring and $M$ is an $S$-semimodule, then $M$ is a subtractive $S$-semimodule if and only if for all $f \in S[X]$ and $g \in M[X]$, there exists an $m \in \mathbb N_0$ such that $$ c(f)^{m+1} c(g) = c(f)^m c(fg),$$ where $m \leq \deg(g)$ and by $c(g)$, we mean the $S$-subsemimodule of $M$ generated by coefficients of $g$. The case $M=S$ is of our special interest.

 While subtractive semirings include rings and bounded distributive lattices and is still a large and important class of semirings, we also show that there are many semirings that Dedekind-Mertens content formula does not hold for them. For example if we consider the idempotent semiring $S = \{ 0,1,u \}$, where $1+u = u+1 = u$, the polynomials $f=1+uX$ and $g=u+X$ of $S[X]$ do not satisfy Dedekind-Mertens content formula, i.e. $c(f)^{m+1} c(g) \not= c(fg)c(f)^m$ and $c(f)^{m+1} c(g) \not= c(fg)c(f)^m$ for all $m\in \mathbb N_0$.

Similar to Gaussian rings (cf. \cite{AC}), we define a semiring $S$ to be Gaussian, if for all $f,g \in S[X]$, we have $c(fg) = c(f)c(g)$. Section 2 of the present paper is devoted to Gaussian semirings. Especially in this section, we prove that every bounded distributive lattice is a Gaussian semiring.

Section 3 is devoted to weak Gaussian semirings, i.e. those semirings that the content formula $c(f)c(g) \subseteq \sqrt {c(fg)}$ holds for all $f,g \in S[X]$. Actually in this section, we prove that if $S$ is a semiring, then the following statements are equivalent:

\begin{enumerate}

\item $S$ is a weak Gaussian semiring,
\item $\sqrt I$ is subtractive for each ideal $I$ of the semiring $S$,
\item Each prime ideal $\textbf{p}$ of $S$ is subtractive.

\end{enumerate}

Later in section 3, we construct a semiring that is a local and a weak Gaussian semiring, i.e. $c(fg) \subseteq c(f)c(g) \subseteq \sqrt {c(fg)}$ for all $f,g \in S[X]$, but still not a subtractive semiring. Note that the radical of an ideal $I$ of a semiring $S$, denoted by $\sqrt I$, is defined as the set $\sqrt I = \{s\in S: \exists n \in \mathbb N (s^n \in I)\}$.

In section 4, we introduce the concept of content semimodules that is a generalization of the concept of content modules introduced in \cite{OR}. We also bring some routine generalizations of the theorems on content modules, though we do not go through them deeply. Actually it was possible to ask the reader to refer to the papers \cite{OR}, \cite{ES} and \cite{R} to see the process of configuring (weak) content algebras from content modules and model them to configure (weak) content semialgebras from content semimodules, but we thought it was a good idea to bring the sketch of the process for the convenience of the reader. Note that the reader who is familiar with content modules, may skip this section without losing the flow.

Let $S$ be a semiring and $M$ an $S$-semimodule. The \textit{content function}, $c$ from $M$ to the ideals of $S$ is defined by
$$ c_M(x) = \bigcap \lbrace I \colon I \text{~is an ideal of~} S \text{~and~} x \in IM \rbrace.$$

$M$ is called a \textit{content $S$-semimodule} if $x \in c_M(x)M $, for all $x \in M$.

 As a matter of fact, it is easy to see that the statement ``$M$ is a content $S$-semimodule'' is equivalent to the statement ``there exists a function $d: M \longrightarrow \Id(S)$ such that $x\in IM$ iff $d(x) \subseteq I$ for all ideals $I$ of $S$''. It is, then, easy to see that $c_M(x)$ is a finitely generated ideal of $S$ for all $x\in M$, if $M$ is a content $S$-semimodule. Therefore if $M$ is a content $S$-semimodule, then the function $c$ is from $M$ to $\FId(S)$, where by $\FId(S)$, we mean the set of finitely generated ideals of $S$.

 Section 5 is devoted to the concept of ``content semialgebras'' as a generalization of polynomials over semirings and applications of Dedekind-Mertens content formula. For the reader's convenience, first we bring the definition of semialgebras and then we explain what we mean by content semialgebras:

Let $S$ and $B$ be two semirings and $\lambda: S \longrightarrow B$ be a semiring homomorphism in this sense that $\lambda$ is a function from $S$ into $B$ with the following properties:

\begin{enumerate}

\item $\lambda (r+s) = \lambda (r) + \lambda (s)$ and $\lambda (rs) = \lambda (r) \lambda (s)$ for all $r,s \in S$;

\item $\lambda (0) = 0$ and $\lambda (1) = 1$.

\end{enumerate}

Then we define the scalar product of $s\in S$ and $f\in B$ as $s \cdot f = \lambda(s)f$. It is easy to check that with the mentioned scalar product, $B$ is an $S$-semimodule such that $s\cdot (fg) = (s \cdot f)g = f(s \cdot g)$ for all $s \in S$ and $f,g \in B$. In this case we say that $B$ is an $S$-semialgebra. Note that if $I$ is an ideal of $S$, then $\lambda (I)$ may not be an ideal of $B$. For this reason, the extension of $I$ in $B$ is defined as the ideal of $B$ generated by $\lambda (I)$ and is denoted by $IB$. One can easily check that $IB = \{ \Sigma_{i=1}^n a_i f_i: a_i \in I, f_i \in B, i\in \mathbb N \}$.

Let $B$ be an $S$-semialgebra. We say that $B$ is a \emph{content} $S$-semialgebra if the homomorphism $\lambda$ is injective (i.e. we can consider $S$ as a subsemiring of $B$) and there exists a function $c: B \longrightarrow \Id(S)$ such that the following conditions hold:

\begin{enumerate}
\item $f\in IB$ iff $c(f) \subseteq I$ for all ideals $I$ of $S$;
\item $c(sf) = sc(f)$ for all $s\in S$ and $f\in B$ and $c(1)= R$;
\item (Dedekind-Mertens content formula) For all $f,g \in B$ there exists an $m\in \mathbb N_0$ such that $c(f)^{m+1} c(g) = c(f)^m c(fg)$.
\end{enumerate}

In section 5, we discuss prime ideals of content semialgebras and we show that in content extensions, minimal primes extend to minimal primes. More precisely, if $B$ is a content $S$-semialgebra, then there is a correspondence between $\Min(S)$ and $\Min(B)$, with the function $ \phi : \Min(S) \longrightarrow \Min(B)$ defined by $\textbf{p} \longrightarrow \textbf{p}B$. As a corollary of this fact, we also show that $\Nil(B) = \Nil(S)B$, where by $\Nil(S)$ we mean the set of all nilpotent elements of the semiring $S$.

In section 6, we define weak content semialgebras and through the continuation of our investigation for prime ideals, we prove that if $S$ is a semiring and $\Lambda , \Delta$ are two index sets such that $\Lambda \cup \Delta \not= \emptyset$. Then the following statements are equivalent:

\begin{enumerate}

\item $c(fg) \subseteq c(f)c(g) \subseteq \sqrt {c(fg)}$, for all $f,g \in S[X_\Lambda , {X_\Delta}^{-1}]$,
\item $\sqrt I$ is subtractive for each ideal $I$ of the semiring $S$,
\item Each prime ideal $\textbf{p}$ of $S$ is subtractive.

\end{enumerate}

In the rest of this section, we work on formal power series over semirings and get a bunch of nice results for them. One of those results is that if $S$ is a Noetherian semiring (i.e. a semiring that each of its ideal is finitely generated), then the following statements are equivalent:

\begin{enumerate}

\item $S[[X_1,X_2,\ldots,X_n]]$ is a weak content $S$-semialgebra,
\item $\sqrt I$ is subtractive for each ideal $I$ of the semiring $S$,
\item Each prime ideal $P$ of $S$ is subtractive.

\end{enumerate}

Sections 7 and 8 are devoted to the investigation of zero-divisors of semiring polynomials and more generally content semialgebras. Most results of these two sections are generalizations of the author's results in the paper \cite{Na}. In fact we believe that the interesting results obtained in sections 1 till 6 are necessary backgrounds to justify what we execute in sections 7 and 8. Now in the following, we describe what we do in these two sections briefly:

An element $s\in S$ is said to be a zero-divisor of the semiring $S$, if there exists a nonzero $s^{\prime}$ such that $ss^{\prime} = 0$. The set of all zero-divisors of the semiring $S$ is denoted by $Z(S)$.

In section 7, we introduce a family of semirings which have very few zero-divisors. It is a well-known result that the set of zero-divisors of a Noetherian semiring is a finite union of its associated primes. Rings having very few zero-divisors have been investigated in \cite{Na}. We define that a semiring $S$ has \textit{very few zero-divisors}, if $Z(S)$ is a finite union of prime ideals in $\Ass(S)$. In this section, we prove that if $S$ is a semiring that has very few zero-divisors and $B$ is a content $S$-semialgebra, then $B$ has very few zero-divisors as well and if the semiring $S$ is weak Gaussian, then the inverse of the recent statement also holds.

Another celebrated property of Noetherian semirings is that every ideal entirely contained in the set of their zero-divisors has a nonzero annihilator. In section 8, we define a semiring $S$ to have \textit{Property (A)}, if each finitely generated ideal $I \subseteq Z(S)$ has a nonzero annihilator. This definition is borrowed from the definition of Property (A) for rings in \cite{HK}. In the last section of the present paper, we prove some results for content semialgebras over semirings having Property (A).

An element $r$ of a ring $R$ is said to be \textit{prime to an ideal} $I$ of $R$ if $I : (r) =I$, where by $I : (r)$, we mean the set of all elements $d$ of $R$ such that $dr \in I$ \cite[p. 223]{ZS}. Let $I$ be an ideal of $R$. We denote the set of all elements of $R$ that are not prime to $I$ by $S(I)$. It is obvious that $r\in S(I)$ iff $r+I$ is a zero-divisor of the quotient ring $R/I$. The ideal $I$ is said to be primal if $S(I)$ forms an ideal and in such a case, $S(I)$ is a prime ideal of $R$. A ring $R$ is said to be \textit{primal}, if the zero ideal of $R$ is primal \cite{Dau}. It is obvious that $R$ is primal iff $Z(R)$ is an ideal of $R$. This motivates us to define primal semirings. We define a semiring $S$ to be primal if $Z(S)$ is an ideal of $S$ and we prove that if $B$ is a content $S$-semialgebra and the content function $c: B \longrightarrow \FId(S)$ is onto, then $B$ is primal iff $S$ is primal and has Property (A). We finish our paper by generalizing this result in the following sense:

 We define a weak Gaussian semiring $S$ to have few zero-divisors, if the set $Z(S)$ of zero-divisors is a finite union of prime ideals. Suppose for the moment that $S$ is a weak Gaussian semiring such that it has few zero-divisors. One can consider $Z(S)= \cup _{i=1}^n \textbf{p}_i$ such that $\textbf{p}_i \nsubseteq \cup _{j=1 \wedge j \neq i}^n \textbf{p}_j$ for all $ 1\leq i \leq n$. Obviously we have $\textbf{p}_i \nsubseteq \textbf{p}_j$ for all $i \neq j$. Also, by using Prime Avoidance Theorem for semirings, it is easy to check that, if $Z(S)= \cup _{i=1}^n \textbf{p}_i$ and $Z(S)= \cup _{k=1}^m \textbf{q}_k$ such that $\textbf{p}_i \nsubseteq \textbf{p}_j$ for all $i \neq j$ and $\textbf{q}_k \nsubseteq \textbf{q}_l$ for all $k \neq l$, then $m=n$ and $\{\textbf{p}_1,\cdots,\textbf{p}_n\}=\{\textbf{q}_1,\cdots,\textbf{q}_n\}$, i.e. these prime ideals are uniquely determined. This is the base for the following definition: We say a weak Gaussian semiring $S$ has \textit{few zero-divisors of degree} $n$, if $S$ has few zero-divisors and $n$ is the number of maximal primes of $Z(S)$. In such a case, we write $\zd(S) = n$ and we prove that if $S$ is a weak Gaussian semiring and $B$ is a content $S$-semialgebra and the content function $c: B \longrightarrow \FId(S)$ is onto, then $\zd(B) = n$ iff $\zd(S) = n$ and $S$ has Property (A). A nice corollary of this theorem is as follows:

 Let $S$ be a subtractive semiring and $\Lambda, \Delta$ be two index sets such that either $\Lambda$ or $\Delta$ is nonempty. Then for all natural numbers $n$, $\zd(S[X_\Lambda , {X_\Delta}^{-1}]) = n$ iff $\zd(S)=n$ and $S$ has Property (A).

Throughout this paper all semirings are commutative with an identity and all semimodules are assumed to be unitary, as we have asserted this from the beginning. Note that \textit{iff} always stands for if and only if.

\section{Dedekind-Mertens Content Formula for Semimodules over Semirings}

Let $S$ be a semiring and $X$ be an indeterminate. The set of all polynomials over the semiring $S$, denoted by $S[X]$, is the set of all formal forms $a_0 + a_1 X + \cdots + a_n X^n$, where $a_0, a_1, \cdots, a_n \in S$. Similar to ring theory, $S[X]$ is a semiring under the usual addition and multiplication of polynomials. In the same way, if $M$ is an $S$-semimodule, one can consider $M[X]$ as an $S[X]$-semimodule under the standard addition and scalar product.

For $g \in M[X]$, we define the content of $g$, denoted by $c(g)$, the $S$-subsemimodule generated by the coefficients of $g$, particulary if $g$ is a monomial, say $g=bX^m$ where $b \in M$ and $m \in \mathbb N_0$, then the content of $g$ is the cyclic $S$-subsemimodule $c(g)=(b)$.

If $g=b_0 + b_1 X + \cdots + b_m X^m$, where $b_0, b_1, \ldots , b_m \in M$ and $b_m \not= 0$, then we say that $g$ is a polynomial of degree $m$ and denote the degree of $g$ by $\deg g = m$. A similar definition works for the degree of polynomials over semirings.

The purpose of this section is to give some content formulas for polynomials over semimodules and semirings. Later we will see the importance and applications of such formulas in the theory of semimodules and semirings.

First we gather some straightforward content formulas in the following proposition only for the sake of reference.

\begin{proposition}

\label{contentformulas}

If $S$ is a semiring and $M$ is an $S$-semimodule and $s\in S$, $m \in M$, $f,f_1,f_2 \in S[X]$ and $g,g_1,g_2 \in M[X]$, then the following formulas hold:

\begin{enumerate}

\item $c(f_1 + f_2) \subseteq c(f_1)+c(f_2)$ and $c(g_1 + g_2) \subseteq c(g_1)+c(g_2)$;
\item $c(f_1 f_2) \subseteq c(f_1) c(f_2)$ and $c(fg) \subseteq c(f)c(g)$;
\item If either $f$ or $g$ is monomial (say $f=sX^n$ or $g=mX^k$), then $c(fg) = c(f)c(g)$, particularly $c(sf) = sc(f)$ and $c(fm)=c(f)m$.

\end{enumerate}

\end{proposition}

 Consider that while the formula $c(fg) = c(f)c(g)$ for $f,g \in S[X]$ is not always true for an arbitrary semiring S, a weaker formula always holds for commutative rings that is called \textit{Dedekind-Mertens content formula} \cite{AG}. We prove the same formula in the following theorem, which is a generalization of Theorem 1 in \cite{AG}. We use essentially the same argument and generalize Dedekind-Mertens content formula for subtractive semimodules and semirings.

\begin{theorem}

\label{dedekindmertens1}

 \textbf{Dedekind-Mertens Lemma for Semimodules}\textbf{.} Let $S$ be a semiring and $M$ be an $S$-semimodule. Then the following statements are equivalent:

 \begin{enumerate}

 \item $M$ is a subtractive $S$-semimodule,

 \item For all $f \in S[X]$ and $g \in M[X]$, if $\deg(g) = m$, then $c(f)^{m+1} c(g) = c(f)^m c(fg)$.

 \end{enumerate}

\begin{proof}

$(1) \rightarrow (2)$: By Proposition \ref{contentformulas}, $c(fg) \subseteq c(f)c(g)$, therefore $c(f)^m c(g) \subseteq c(f)^{m+1} c(g)$. We prove that $c(f)^{m+1} c(g) \subseteq c(f)^m c(fg)$ by induction on $m$ and $n$, the degrees of $g$ and $f$ respectively.

If $f$ is a monomial; say $f=a_n X^n$ and $g=b_0 + b_1 X + \cdots + b_m X^m$, then from Proposition 1, we get $c(f)^m c(fg) = c(f)^{m+1} c(g)$. For the same reason, if $g$ is a monomial, say $g= b_m X^m$ and $f= a_0 + a_1 X + \cdots + a_n X^n$, then again by Proposition 1, we have the same result: $c(f)^m c(fg) = c(f)^{m+1} c(g)$. This means that if either $f$ or $g$ is a monomial and as a particular case, if either $f$ or $g$ has degree zero, then the theorem is true.

By induction we may suppose the following statements:

\begin{enumerate}

\item[(I)] If $\deg f =n < r$ and $\deg g = m$, then $c(f)^{m+1} c(g) \subseteq c(f)^m c(fg)$.
\item[(II)] If $\deg f =r$ and $\deg g = m < s$, then $c(f)^{m+1} c(g) \subseteq c(f)^m c(fg)$.

\end{enumerate}

Now let $f= a_0 + a_1 X + \cdots + a_r X^r$, $a_r \not= 0$, $g=b_0 + b_1 X + \cdots + b_s X^s$, $b_s \not= 0 $, and assume that
neither $f$ nor $g$ is a monomial. Then we wish to prove that $c(f)^{s+1} c(g) \subseteq c(f)^s c(fg)$.

Let $f_1=a_0 + a_1 X + \cdots + a_{r-1}X^{r-1}$, $g_1=b_0 + b_1 X + \cdots + b_{s-1}X^{s-1}$, $h=fg$, $h_1=f_1 g$ and $h_2=f g_1$. It is, then, clear that we have the following:

\begin{enumerate}

\item $h = \sum_{k=0}^{r+s} c_k X^k$, where $c_k = \sum_{i+j=k} a_i b_j$, for $0 \leq k \leq r+s$;

\item $h_1 = \sum_{k=0}^{r+s-1} {c_k}^{(1)} X^k$, where ${c_k}^{(1)}= c_k$ for $0 \leq k \leq r-1$ and ${c_k}^{(1)} + a_r b_{k-r} = c_k$ for $r \leq k \leq r+s-1$.

\item $h_2 = \sum_{k=0}^{r+s-1} {c_k}^{(2)} X^k$, where ${c_k}^{(2)} = c_k$ for $0 \leq k \leq s-1$ and ${c_k}^{(2)} + a_{k-s} b_s = c_k$ for $s \leq k \leq r+s-1$.

\end{enumerate}

We claim that $c(f_1 g) \subseteq c(fg)+a_r c(g_1)$. The proof is as follows:

Obviously $c(f_1 g) = c(h_1) = ({c_0}^{(1)}, \ldots, {c_{r+s-1}}^{(1)})$ and $ c(fg)+a_r c(g_1) = (c_0, \ldots , c_{r+s})+(a_r)(b_0, \cdots , b_{s-1})$ and in order to prove that $c(f_1 g) \subseteq c(fg)+a_r c(g_1)$, we just need to prove that ${c_k}^{(1)} \in (c_0, \ldots , c_{r+s})+(a_r)(b_0, \ldots , b_{s-1})$ for all $0 \leq k \leq r+s-1$.

If $0 \leq k \leq r-1$, then ${c_k}^{(1)}= c_k$ and therefore there is nothing to prove. Let for the moment $r \leq k \leq r+s-1$, so it is clear that $c_k, a_r b_{k-r} \in c(fg)+a_r c(g_1)$. On the other hand $c(fg)+a_r c(g_1)$ is a subtractive $S$-subsemimodule of $M$. Hence, from ${c_k}^{(1)} + a_r b_{k-r} = c_k$, we get that ${c_k}^{(1)} \in c(fg)+a_r c(g_1)$. In a similar way, we can easily prove that $c(fg_1) \subseteq c(fg)+ c(f_1) b_s$.

Since $c(f)^{s+1} c(g)$ is generated by elements of the form $\alpha = {a_0}^{n_0} {a_1}^{n_1} \cdots {a_r}^{n_r} b_i$, where $\sum_{j=0}^r n_j = s+1$ and $0 \leq i \leq s$, it suffices to show that each element of this form is contained in $c(f)^s c(fg)$. If $n_r \not= 0$ and $i=s$, then $\alpha = {a_0}^{n_0} {a_1}^{n_1} \cdots {a_r}^{n_r -1} c_{r+s}$, since $c_{r+s}=a_r b_s$. But $c_{r+s} \in c(fg)$, so that $\alpha \in c(f)^n c(fg)$. If $n_r \not= 0$ and $i<s$, then $\alpha \in c(f)^s (a_r)c(g_1)$. In the case $n_r=0$, we have $\alpha \in c(f_1)^{s+1} c(g)$. Therefore,
$$c(f)^{s+1}c(g) \subseteq c(f)^s c(g) + c(f_1)^{s+1} c(g) + c(f)^s (a_r)c(g_1).$$

By (I), $c(f_1)^{s+1} c(g) \subseteq c(f_1)^s c(f_1 g)$. $$$$ On the other hand we have seen that $c(f_1 g) \subseteq c(fg)+ a_r c(g_1)$. Consequently, we have
$$ c(f_1)^{s+1} c(g) \subseteq c(f_1)^s c(fg) + c(f_1)^s (a_r) c(g_1) \subseteq c(f)^s c(fg) + c(f)^s (a_r) c(g_1). $$

Thus, $c(f)^{s+1}c(g) \subseteq c(f)^s c(g) + c(f)^s (a_r)c(g_1)$.
$$$$
By (II), if $\lambda = \deg g_1$, then $c(f) ^{\lambda +1}c(g_1) \subseteq c(f)^{\lambda} c(fg_1)$. Since $\lambda \leq s-1$, we have $ c(f)^s c(g_1) \subseteq c(f)^{s-1} c(fg_1)$. But we have also seen that $c(fg_1) \subseteq c(fg)+ c(f_1) b_s$ so that
$$ c(f)^s (a_r) c(g_1) \subseteq c(f)^{s-1}(a_r)c(fg) + c(f)^{s-1} c(f_1) (a_r)(b_s) $$
$$\subseteq c(f)^s c(fg) + c(f)^{s-1} c(f_1) (c_{r+s}) \subseteq c(f)^s c(fg).$$

Therefore, $c(f)^{s+1} c(g) \subseteq c(f)^s c(fg)$, as we wished to prove. It now follows by induction that if $\deg g = m$, then $c(f)^{m+1} c(g) \subseteq c(f)^m c(fg)$. Hence, $c(f)^{m+1} c(g) = c(f)^m c(fg)$ and the proof of $(1) \rightarrow (2)$ is complete.

$(2) \rightarrow (1)$: Let $N$ be an $S$-subsemimodule of $M$. Take $u,v \in M$ such that $u+v, u \in N$. Define $f=1+X$ and $g=u+vX+uX^2$. It is easy to see that $fg=u+(u+v)X+(u+v)X^2+uX^3$, $c(f)=S$, $c(g)=(u,v)$ and $c(fg)=(u+v,u)$. But according to our assumption, Dedekind-Mertens content formula holds and therefore there exists an $m\in \mathbb N_0$ such that $c(f)^{m+1} c(g) = c(f)^m c(fg)$. This means that $(u,v) = (u+v,u)$, which implies $v\in (u+v,u) \subseteq N$ and $N$ is subtractive.
\end{proof}

\end{theorem}

In the above theorem if we suppose $M=S$, we get the following:

\begin{theorem}

\label{dedekindmertens2}

 \textbf{Dedekind-Mertens Lemma for Semirings}\textbf{.} Let $S$ be a semiring. Then the following statements are equivalent:

 \begin{enumerate}

 \item $S$ is a subtractive semiring,
 \item If $f, g \in S[X]$ and $\deg(g) = m$, then $c(f)^{m+1} c(g) = c(f)^m c(fg)$.

 \end{enumerate}

\end{theorem}

Our next task is to prove Dedekind-Mertens content formula for polynomials with finitely many indeterminates. To do this, first we introduce the concept of the support of a polynomial. Let $f$ be a polynomial, the support of $f$, denoted by $\supp (f)$ is the set of all nonzero coefficients of $f$. Now we show a nice result that is a generalization of Lemma 2 in \cite{AG}. We use a similar technique that is given in the proof of the Lemma 2 in \cite{AG}.

\begin{lemma}

\label{supppoly}

Let $M$ be an $S$-semimodule and $X_1, \ldots , X_n, Y$ be $n+1$ distinct indeterminates over the semiring $S$ and the semimodule $M$ and $f \in S[X_1, \ldots , X_n, Y]-\{0\}$ and $g \in M[X_1, \ldots , X_n, Y]-\{0\}$ be polynomials. Then there exist $f^* \in S[X_1, \ldots , X_n]-\{0\}$ and $g^* \in M[X_1, \ldots , X_n]-\{0\}$ such that $\supp (f) = \supp (f^*)$, $\supp (g) = \supp (g^*)$ and $\supp (fg) = \supp (f^* g^*)$.
\end{lemma}

\begin{proof}

Let $h(X_1, \ldots , X_n, Y)$ be a nonzero polynomial of $n+1$ indeterminates. We write $h$ as a polynomial in $Y$ with coefficients $h_0, \ldots, h_k$ to be polynomials of $n$ indeterminates $X_1, \ldots, X_n$, i.e. $h = \sum_{i=0}^k h_i Y^i$, and we denote the degree of $h$ in $X_n$ by $\deg_n h$, which is equal to the maximum of the degrees of the $h_i$'s in $X_n$. We observe that if $m > \deg_n h$, then the coefficients of $h^*(X_1, \ldots, X_n, {X_n}^m) = \sum_{i=0}^k h_i X_{n}^{mi}$ are the same as the coefficients of $h$, because if $0 \leq i \leq k$ and $h_i \not= 0$, then $\deg_n h_i < m $ so that $mi \leq \deg_n \beta < m(i+1)$ for any nonzero monomial $\beta$ of $h_i {X_n}^{mi}$. Therefore, the nonzero monomials appearing in $h_i {X_n}^{mi}$ are distinct from those appearing in $h_j {X_n}^{mj}$ for $i \not= j$. It follows that $h$ and $h^*$ have the same support.

We choose $m = deg_n f + deg_n g + 1$. Then $m > \deg_n f$, $m > \deg_n g$ and $m > \deg_n fg$. It is easy to see that $(fg)^* = f^* g^*$ where by $l^*$, we mean the mapping $$*: h(X_1, \ldots , X_n, Y) \longrightarrow h^*=h(X_1, \ldots , X_n, {X_n}^m)$$ from $M[X_1, \ldots , X_n, Y]$ ($S[X_1, \ldots , X_n, Y]$) onto $M[X_1, \ldots , X_n]$ ($S[X_1, \ldots , X_n]$). From all we said we conclude that $\supp (f) = \supp (f^*)$, $\supp (g) = \supp (g^*)$ and $\supp (fg) = \supp (f^* g^*)$ and the proof is complete.
\end{proof}

Now we bring a generalization of Dedekind-Mertens content formulas for polynomials with finitely many indeterminates. Note that if $M$ is an $S$-semimodule and $g \in M[X_1, X_2, \ldots , X_n]$ is a polynomial with $n$ indeterminates, then the content of $g$, denoted by $c(g)$, is the $S$-subsemimodule of $M$, generated by its coefficients.

\begin{theorem}

\label{dedekindmertens3}

Let $S$ be a semiring and $M$ be an $S$-semimodule. Then $M$ is a subtractive $S$-semimodule iff for all $f \in S[X_1, \ldots, X_n]$ and $g \in M[X_1, \ldots, X_n]$, there exists an $m \in \mathbb N_0$ such that $c(f)^{m+1} c(g) = c(f)^m c(fg)$.

\end{theorem}

\begin{proof}
$(\rightarrow)$: The proof is by induction. For $n=1$, the result follows from Theorem \ref{dedekindmertens1}. If the result is true for $n=k$ and $f\in S[X_1, \ldots, X_k, X_{k+1}]$ and $g\in M[X_1, \ldots, X_k, X_{k+1}]$, then without loss of generality, we may assume that both $f$ and $g$ are nonzero polynomials and Lemma \ref{supppoly} implies the existence of $f^* \in S[X_1, \ldots, X_k]$ and $g^* \in M[X_1, \ldots, X_k]$ such that $c(f) = c(f^*)$,  $c(g) = c(g^*)$ and $c(fg)=c(f^* g^*)$. From the induction assumption, there exists an $m \in \mathbb N_0$ such that $c(f^*)^{m+1} c(g^*) = c(f^*)^m c(f^* g^*)$ and the proof is complete.

The proof of $(\leftarrow)$ is similar to the proof of $(2) \rightarrow (1)$ in Theorem \ref{dedekindmertens1}.
\end{proof}

Let for the moment $\Lambda$ and $\Delta$ be index sets. Put $M[X_\Lambda , {X_\Delta}^{-1}]$ to be the union of all Laurent polynomials $M[{X_\lambda}_1, \ldots, {X_\lambda}_n, {{X_\delta}_1}^{-1}, \ldots, {{X_\delta}_m}^{-1}]$, where ${\lambda}_1, \ldots, {\lambda}_n \in \Lambda$, ${\delta}_1, \ldots, {\delta}_m \in \Delta$ and $n,m \in \mathbb N$. We have the following:

\begin{theorem}

\label{dedekindmertens4}

 Let $S$ be a semiring, $M$ be an $S$-semimodule and $\Lambda , \Delta$ be two index sets such that $\Lambda \cup \Delta \not= \emptyset$. Then $M$ is a subtractive $S$-semimodule iff for all $f \in S[X_\Lambda , {X_\Delta}^{-1}]$ and $g \in M[X_\Lambda , {X_\Delta}^{-1}]$, there exists an $m \in \mathbb N_0$ such that $c(f)^{m+1} c(g) = c(f)^m c(fg)$.

\end{theorem}

According to Dedekind-Mertens content formula, it is natural to ask when $c(fg) = c(f)c(g)$ holds for all $f,g\in S[X]$, where $S$ is a semiring. The next section is devoted to this question.

\section{Gaussian Semirings}

Let, for the moment, $R$ be a commutative ring with identity. By definition, the ring $R$ is said to be Gaussian, if $c(fg)=c(f)c(g)$, for all $f,g \in R[X]$ (\cite{AC}). Similarly we define Gaussian semirings:

\begin{definition}

A semiring $S$ is called Gaussian if $c(fg)=c(f)c(g)$ for all polynomials $f,g \in S[X]$.

\end{definition}

Note that by Theorem \ref{dedekindmertens2}, a Gaussian semiring $S$ needs to be subtractive. In this section, we give some of conditions that cause a subtractive semiring to be Gaussian. A well-known theorem in commutative ring theory states that if every finitely generated ideal of a ring $R$ is principal, then the ring $R$ is Gaussian (\cite{AC}). In the following we show that a semiring $S$ is Gaussian if every finitely generated ideal of the semiring $S$ is principal generated by the sum of its generators. One of the interesting corollaries of this theorem is that every bounded distributive lattice is a Gaussian semiring.

\begin{theorem}

\label{GaussianSemiring}

 Let $S$ be a semiring such that $(a_1, a_2, \ldots, a_n) = (a_1+a_2+\cdots +a_n)$ for all elements $a_1, a_2, \ldots, a_n \in S$. Then $S$ is a Gaussian semiring.

\begin{proof}
Consider the polynomials $f,g \in S[X]$ such that $f=a_0+a_1 X+ \cdots + a_n X^n$ and $g=b_0+b_1 X+ \cdots + b_m X^m$. According to the definition of multiplication in polynomial semirings, $fg=c_0 + c_1 X + \cdots + c_{n+m} X^{n+m}$, where $c_k = a_0 b_k + \cdots + a_i b_{k-i} + \cdots + a_k b_0$. Now observe that $c(fg)= (c_0, c_1 , \ldots , c_{n+m}) = (c_0 + c_1 + \cdots + c_{n+m})= ((a_0 + a_1 + \cdots + a_n)(b_0 + b_1 + \cdots + b_m))= (a_0 + a_1 + \cdots + a_n)(b_0 + b_1 + \cdots + b_m)=c(f)c(g)$ and the proof is complete.
\end{proof}

\end{theorem}

\begin{theorem}

\label{bdl}

Every bounded distributive lattice is a Gaussian semiring.

\begin{proof}

Let $(L=[0,1],+,\cdot)$ be a bounded distributive lattice. Obviously $(L,+,\cdot,0,1)$ is a semiring. Now we show that $(a_1,a_2,\ldots,a_n) = (a_1+a_2+\cdots +a_n)$. Since $a_i=a_i \cdot (a_1+a_2+ \cdots + a_n)$, vividly $a_i \in (a_1+a_2+ \cdots + a_n)$ and therefore $a_1, a_2 , \ldots a_n \in (a_1+a_2+ \cdots + a_n)$. Obviously $a_1+a_2+ \cdots + a_n \in (a_1, a_2, \ldots , a_n)$.
\end{proof}

\end{theorem}

Examples of bounded distributive lattices consists of Boolean semiring $(\{0,1\},+,\cdot)$, power set lattice $(P(A), \cup, \cap)$, bottleneck semiring $(\mathbb R \cup \{ - \infty, + \infty \}, \max, \min)$ and fuzzy semiring $(\mathbb I = [0,1], \max, \min)$.

An $R$-module $M$ is said to be distributive if $A \cap (B+C) = (A \cap B) + (A \cap C)$ for all $A,B,C \in \Sub_R(M)$, where by $\Sub_R(M)$ we mean the set of all $R$-submodules of $M$ (\cite{C}). One can easily see that if $M$ is a distributive $R$-module, then $(\Sub_R(M), +, \cap)$ is a bounded distributive lattice and another interesting example for Theorem \ref{bdl}.

A semiring $(S,+,\cdot ,0,1)$ is called zerosumfree if $a+b=0$ implies $a=b=0$ for all $a,b \in S$. A semiring $F$ is said to be a semifield if $F$ is a semiring that any nonzero element of $F$ has a multiplicative inverse, i.e. for each $s\in F-\{0\}$ there exists an $s^{-1}$ such that $ss^{-1}=1$.

Let $X$ be a nonempty set and $S$ be a semiring. Over the set of all functions from $X$ to $S$, denoted by $F(X,S)$, we define addition and multiplication as $(f+g)(x) = f(x)+g(x)$ and $(fg)(x) = f(x)g(x)$. One can easily check that $F(X,S)$ with the mentioned addition and multiplication is a semiring. The following example is another interesting corollary of Theorem \ref{GaussianSemiring}.

\begin{example}

If $X$ is a nonempty set and $S$ is a zerosumfree semifield, then $F(X,S)$ is a Gaussian semiring.

\begin{proof}

We show that if $a_1,a_2,\ldots,a_n \in F(X,S)$ then $(a_1,a_2,\ldots,a_n) = (a_1+a_2+\cdots +a_n)$. For doing so, we just need to prove that $a_i \in (a_1+a_2+\cdots +a_n)$. We define $b_i \in F(X,S)$ in the following way:

$b_i(x) = a_i(x) / a_1(x)+a_2(x)+\cdots+a_n(x)$ if $a_i(x) \not= 0$ and $b_i(x)=0$ if $a_i(x)=0$. Note that $a_1(x)+a_2(x)+\cdots+a_n(x) \not= 0$ if $a_i(x) \not= 0$ for all $x\in X$. It is, then, easy to see that $b_i (a_1+a_2+\cdots+a_n) = a_i$ and therefore $a_i \in (a_1+a_2+\cdots +a_n)$.
\end{proof}

\end{example}

\begin{theorem}

Let $(S,+,0,<)$ be a totally ordered commutative monoid such that $0$ is the least element of $S$ and for all $a,b \in S$, if $a \leq b$, then there exists an $x \in S$ such that $a+x=b$. Then $(S \cup \{ +\infty \}, \min, +)$ is a Gaussian semiring.

\end{theorem}

\begin{proof}

For the simplicity, we let $R=S \cup \{ +\infty \}$. We claim that for $a_1, a_2, \ldots , a_n \in R$, we have $$(a_1, a_2, \ldots , a_n) = [\min \{a_1, a_2, \ldots , a_n \} , +\infty],$$ where by the interval $[c,d]$, we mean the set of all $x \in R$ such that $c \leq x \leq d$.

By definition, $(a_1, a_2 , \ldots , a_n) = \{ \min \{r_1 + a_1, r_2 + a_2, \ldots , r_n + a_n \} : r_1,r_2,\ldots,r_n \in R\}$.

Put $I = (a_1, a_2, \ldots , a_n)$ and let $a_k=\min \{a_1, a_2, \ldots , a_n \}$ and assume that $r_k \in R$ and $r_1 = \cdots = r_{k-1} = r_{k+1} = \cdots = r_n = +\infty $. Then it is obvious that $\min \{r_1 + a_1, r_2 + a_2, \ldots , r_n + a_n \} = r_k + a_k$. Therefore $I \subseteq [a_k, +\infty]$. On the other hand let $ s \in [a_k, +\infty]$, so there exists $x \in R$ such that $a_k + x = s$ and therefore $s \in I$.

But $[\min \{a_1, a_2, \ldots , a_n \} , +\infty]$ is the principal ideal $(\min \{a_1, a_2, \ldots , a_n \})$ of $R$. Therefore according to Theorem \ref{GaussianSemiring}, $R$ is a Gaussian semiring and the proof is complete.
\end{proof}

\begin{corollary}
The tropical semiring $(\mathbb N_0 \cup \{+\infty\}, \min , +, + \infty, 0)$ is a Gaussian semiring.
\end{corollary}

We end this section by bringing two well-known family of Gaussian semirings inspired by their ring versions. Note that a semiring $(S,\textbf{m})$ is said to be a local semiring if the ideal $\textbf{m}$ is the only maximal ideal of the semiring $S$.

\begin{proposition}

\label{GaussianSemiring2}

Let $(S,\textbf{m})$ be a local and subtractive semiring with $\textbf{m} ^2 =0$, then $S$ is a Gaussian semiring.

\begin{proof}
 Let $f,g \in S[X]$ such that $c(f) \subseteq \textbf{m}$ and $c(g) \subseteq \textbf{m}$, then $c(fg) \subseteq c(f)c(g) = (0)$, otherwise one of them, say $c(f)$, is equal to $S$ and according to Dedekind-Mertens content formula, we have $c(fg) = c(g) = c(f)c(g)$.
\end{proof}

\end{proposition}

For some examples of semirings satisfying the conditions of Proposition \ref{GaussianSemiring2}, but are not rings, refer to Proposition \ref{WeakGaussian1}.

Let $S$ be a subtractive semiring, then by Dedekind-Mertens content formula, for $f,g \in S[X]$, there exists an $m \in \mathbb N_0 $ such that $c(f)^m c(fg) = c(f)^m c(f)c(g)$. Therefore if we can cancel $c(f)^m$ from both sides of the formula, then we reach to $c(fg) = c(f)c(g)$. This can be a good motivation to consider cancelation ideals. We bring the definition of cancelation ideals from \cite{L} as follows:

\begin{definition}

A nonzero ideal $I$ of the semiring $S$ is said to be cancelation if $IJ=IK$ implies $J=K$ for all ideals $J,K$ of $S$.

\end{definition}

For example the principal ideal $(a)$ is cancelation if $a$ is a nonzero cancelation element of the semiring $S$, i.e. $ab=ac$ implies $b=c$ for all elements $b,c \in S$.

From this discussion, we have the following:

\begin{proposition}

Let $S$ be a subtractive semiring. If every nonzero finitely generated ideal of the semiring $S$ is cancelation, then $S$ is a Gaussian semiring.

\end{proposition}

Note that not all subtractive semirings are Gaussian as the following example mentioned in \cite{A} shows:

\begin{example}

\label{nonGaussian}

Let $R = \mathbb Z + 2i \mathbb Z$ and $f=g=2i+2X$. Then it is easy to check that $c(fg) =(4)$, while $c(f)c(g)= (4,4i)$ and therefore $R = \mathbb Z + 2i \mathbb Z$ is not Gaussian. Obviously $R$ is a subtractive semiring, since $R$ is a ring.

\end{example}

\section{Weak Gaussian Semirings}

In the previous section, we gave a bunch of families of Gaussian semirings. Particularly we saw that all bounded distributive lattices are examples of Gaussian semirings. On the other hand, let us remind that $c(f)c(g) \subseteq \sqrt {c(fg)}$ holds for all $f,g \in R[X]$, if $R$ is an arbitrary commutative ring (refer to \cite[Exercise 3.4]{E} and for a generalization of that refer to \cite{R}). In fact it is easy to see that by Theorem \ref{dedekindmertens2} the content formula $c(f)c(g) \subseteq \sqrt {c(fg)}$ holds for all $f,g \in S[X]$, if $S$ is a subtractive semiring, where by the radical of an ideal $I$ of a semiring $S$, denoted by $\sqrt I$, we mean the set $\sqrt I = \{s\in S: \exists n \in \mathbb N (s^n \in I)\}$. Actually there are some non-subtractive semirings that this content formula does not hold for them. We give the first example of such semirings in the following:

\begin{example}

\label{LaGrassaexample}

Consider the idempotent semiring $S = \{ 0,u,1 \}$, where $1+u = u+1 = u$ (\cite{L}). Put $f=1+uX$ and $g=u+X$. It is easy to see that $fg = (1+uX)(u+X) = u+uX+uX^2$, $c(fg) = \{0,u\}$ and $c(f)c(g)=S$ while $\sqrt {c(fg)} = \sqrt {\{0,u\}} = \{0,u\}$ and this means that $c(f)c(g) \nsubseteq \sqrt {c(fg)}$.

\end{example}

From all we said we are inspired to give the following definition:

\begin{definition}

Let $S$ be a semiring. We say the semiring $S$ is a weak Gaussian semiring, if $c(f)c(g) \subseteq \sqrt {c(fg)} $ for all $f,g \in S[X]$.

\end{definition}

An immediate consequence of the above definition is that if $S$ is a weak Gaussian semiring and $\sqrt I = I$ for any ideal $I$ of the semiring $S$, then $S$ is a Gaussian semiring. For example let $C=\{0,u,1\}$ be a chain such that $0<u<1$. We define $a \oplus b = u$ if $a=b=1$, otherwise $a \oplus b = \max \{a,b\}$. One can check that $(C,\oplus,\min,0,1)$ is a subtractive semiring with the set of ideals $\Id(C)=\{\{0\}, \{0,u\}, C \}$. Therefore $C$ is a weak Gaussian semiring. On the other hand for each ideal $I$ of $C$, we have $\sqrt I = I$, so $c(fg) \subseteq c(f)c(g) \subseteq \sqrt {c(fg)} = c(fg)$ and this means that $C$ is, in fact, a Gaussian semiring.

Now we are in the position to prove the following important theorem:

\begin{theorem}

\label{WDMcriteria1}

Let $S$ be a semiring. Then the following statements are equivalent:

\begin{enumerate}

\item $S$ is a weak Gaussian semiring,
\item $\sqrt I$ is subtractive for each ideal $I$ of the semiring $S$,
\item Each prime ideal $\textbf{p}$ of $S$ is subtractive.

\end{enumerate}

\begin{proof}
$(1) \rightarrow (2)$: Suppose $I$ is an ideal of $S$ and $a,b \in S$ such that $a+b, a \in \sqrt I$. We need to show that $b\in \sqrt I$. Let $X$ be an indeterminate over the semiring $S$ and put $f=a+bX$ and $g=b+(a+b)X$. We have $fg= ab+(a^2+ab+b^2)X+(ab+b^2)X^2$ and so $c(fg) \subseteq \sqrt I$. Since $c(f)c(g) \subseteq \sqrt {c(fg)}$, we have $c(f)c(g) \subseteq \sqrt I$. But $c(f)=(a,b)$ and $c(g)=(a+b,b)$. So $b^2 \in \sqrt I$ and therefore $b \in \sqrt I$.

$(2) \rightarrow (3)$ is obvious.

$(3) \rightarrow (1)$ Let each prime ideal $\textbf{p}$ of $S$ be subtractive. We need to show that $c(f)c(g) \subseteq \sqrt {c(fg)}$, for all $f,g \in S[X]$. Let $f,g \in S[X]$ and suppose that $\textbf{p}$ is a prime ideal of $S$ and $c(fg) \subseteq \textbf{p}$. Obviously $fg \in \textbf{p}[X]$. But by \cite[Theorem 2.6]{L} - which says that if $\textbf{p}$ is an ideal of a semiring $S$ and $X$ is an indeterminate over $S$, then $\textbf{p}[X]$ is a prime ideal of $S[X]$ iff $\textbf{p}$ is a subtractive prime ideal of $S$ - the ideal $\textbf{p}[X]$ is a prime ideal of $S[X]$ and so either $f\in \textbf{p}[X]$ or $g\in \textbf{p}[X]$ and this means that either $c(f) \subseteq \textbf{p}$ or $c(g) \subseteq \textbf{p}$ and in any case $c(f)c(g) \subseteq \textbf{p}$. Consequently by \cite[Proposition 7.28 (Krull's Theorem)]{G} - that says that $ \sqrt I = \bigcap_{\textbf{p}\in \Spec_I(S)} \textbf{p}$, where by $\Spec_I(S)$ we mean the set of all prime ideals of $S$ containing $I$  - we have $c(f)c(g) \subseteq \bigcap_{\textbf{p}\in \Spec_{c(fg)}(S)} \textbf{p} = \sqrt {c(fg)}$ and the proof is complete.
\end{proof}

\end{theorem}

From all we have seen until now, we know that all subtractive semirings are weak Gaussian. We also know that the non-subtractive semiring mentioned in Example \ref{LaGrassaexample} is not a weak Gaussian semiring. The question may arise if there is a weak Gaussian semiring that is not subtractive? Actually in Proposition \ref{WeakGaussian2}, we give a family of weak Gaussian semirings that are not subtractive. In order to do that let us remind that if $(M,+,0)$ is a commutative monoid, then a subset $N$ of $M$ is said to be a submonoid of $M$ if $0\in N$ and $m+n \in N$ for all $m,n \in N$. A submonoid $N$ is said to be subtractive if $m+n,m\in N$ implies $n\in N$ for all $m,n\in M$. A monoid $M$ is subtractive if all its submonoids are subtractive. Finally the monoid $(M,+,0)$ is said to be idempotent if $x+x=x$ for all $x\in M$.

\begin{proposition}

\label{WeakGaussian1}

Let $(P,+,0)$ be an idempotent commutative monoid and set $S=P \cup \{1\}$. Let extend addition on $S$ as $a+1=1+a=1$ for all $a\in S$ and define multiplication over $S$ as $ab=0$ for all $a,b \in P$ and $a\cdot 1=1 \cdot a=a$ for all $a\in S$. Then the following statements hold:

\begin{enumerate}

\item $(S,+,\cdot ,0,1)$ is a semiring and $P$ is the only maximal ideal of the semiring $S$ with $P^2=(0)$;
\item $S$ is a weak Gaussian semiring;
\item $I \not= S$ is a subtractive ideal of $S$ iff $I$ is a subtractive submonoid of $P$ for all $I \subseteq P$;
\item $S$ is a subtractive semiring iff $P$ is a subtractive monoid;
\item $S$ is a Gaussian semiring iff $P$ is a subtractive monoid.

\end{enumerate}

\begin{proof}
$(1)$, $(3)$ and $(4)$ are straightforward. We only prove $(2)$ and $(5)$. For proving $(2)$, first we prove that the only prime ideal of $S$ is $P$. Suppose $Q$ is a prime ideal of $S$. By definition of multiplication, $s^2=0$ for all $s\not=1$, so $s^2\in Q$ for all $s\not=1$. But $Q$ is prime and therefore $s\in Q$ and this means that $Q=P$ is the only prime ideal of $S$. One can easily see that $P$ is a subtractive ideal of $S$ and by Theorem \ref{WDMcriteria1}, $S$ is a weak Gaussian semiring. By considering Proposition \ref{GaussianSemiring2}, $(5)$ is obtained from $(4)$.
\end{proof}

\end{proposition}

\begin{proposition}

\label{WeakGaussian2}

Let $(T=[0,1],<)$ be a chain with the least element $0$ and the most element $1\not=0$. We define addition ``$+$'' on $T$ as $a+b = \max\{a,b\}$ and multiplication ``$\cdot$'' as $a\cdot b=0$ if $a<1$ and $b<1$ and $a\cdot 1=1 \cdot a=a$ for all $a,b\in T$. Then the following statements hold:

\begin{enumerate}

\item $(T,+,\cdot ,0,1)$ is a semiring and $\textbf{m}=T-\{1\}$ is the only maximal ideal of $T$ with $\textbf{m}^2=(0)$;
\item $T$ is a weak Gaussian semiring;
\item If $T$ has at most three elements, then $T$ is a Gaussian semiring;
\item If $T$ has at least four elements, then $T$ is not a subtractive semiring.

\end{enumerate}

\begin{proof}

$(1)$ is straightforward and $(2)$ is obtained from Proposition \ref{WeakGaussian1}. We prove $(3)$ and $(4)$ as follows:

$(3)$: If $T$ has only two elements then $T = \mathbb B =\{0,1\}$ is a bounded distributive lattice and by Theorem \ref{bdl}, a Gaussian semiring. If $T=\{0,u,1\}$, then ideals of $T$ are $(0)$, $\{0,u\}$ and $T$ and each ideal of $T$ is subtractive. But $\textbf{m}=\{0,u\}$ is the only maximal ideal of $T$ such that $\textbf{m}^2=(0)$ and therefore according to Proposition \ref{GaussianSemiring2}, $T$ is a Gaussian semiring.

$(4)$: Now let $T$ have at least four elements. So we can choose $a,b\in T$ such that $0<a<b<1$. Now consider the polynomials $f=1+X$ and $g=b+aX+bX^2$. It is easy to see that $fg= b+bX+bX^2+bX^3$ and $c(f)=T$, $c(g)=\{0,a,b\}$ and $c(fg)=\{0,b\}$. Hence $c(f)^{m+1} c(g) = \{0,a,b\}$, while $c(f)^m c(fg) =  \{0,b\}$ for all $m\in \mathbb N_0$ and by Theorem \ref{dedekindmertens2}, this means that $T$ is not a subtractive semiring.
\end{proof}

\end{proposition}

We end this section by giving a couple of more examples of semirings that are not weak Gaussian.

\begin{example} Examples of semirings that are not weak Gaussian.

\label{noWDM}

\begin{enumerate}

\item Note that the set $R=[1, + \infty) \cup \{0\} = \{x\in \mathbb R: 1 \leq x < + \infty\} \cup \{0\}$ with ordinary addition and multiplication of real numbers is a semiring. Let $S$ be a subsemiring of $R$ and set $P=S-\{1\}$. Then $P$ is a prime ideal of $S$, since if $a,b \notin P$, then $a=b=1$ and therefore $ab=1$, which means that $ab\notin P$. Let $a \in P-\{0\}$. So $a+1 \not= 1$ and therefore $a+1 \in P$, while $1\notin P$. Hence $P$ is not subtractive and $S$ is not weak Gaussian. From this, we get that the most popular semiring, i.e. the semiring of non-negative integers $(\mathbb N_0, +,\cdot ,0,1)$ is not weak Gaussian; something that may not be expected at first sight!

\item Consider the Truncation semiring, i.e. the semiring $(T_k, \max, \min \{a+b,k\}, - \infty, 0)$, where $1 \leq k$ and $T_k = \{ - \infty, 0, 1, \ldots, k\}$. For simplicity, we set $a \oplus b = \max \{a,b\}$, $a \odot b = \min \{a+b,k\}$, $a^1 = a$ and $a^{n+1} = a^n \odot a$. Let $I$ be an ideal of $T_k$ such that $I \not= \{- \infty \}$ and $I \not= T_k$. One can easily check that for each $a \in I-\{ - \infty \}$, there exist an $1 \leq m \leq k$ such that $a^m = k$. This means that $k \in I$ and $\sqrt I = T_k - \{ 0 \}$. Consider that if $a \in I-\{ - \infty \}$, then $a \oplus 0 \in I$ and $a\in I$, while $0 \notin I$. Therefore $\sqrt I$ is not subtractive and according to Theorem \ref{WDMcriteria1}, $T$ is not a weak Gaussian semiring.

\item Let $i<n$ be positive integers. The semiring $(B(n,i) = \{0,1,\ldots,n-1 \},\oplus,\odot,0,1)$ is defined as follows:

The addition $\oplus$ is defined as $x\oplus y = x+y$ if $x+y \leq n-1$ and $x\oplus y = l$ if $x+y > n-1$ where $l$ is the unique number satisfying the conditions $i \leq l \leq n-1$ and $l \equiv_{\mod (n-i)} x+y $ and multiplication $\odot$ is defined similarly. Our claim is that if $i>1$, then $B(n,i)$ is not a weak Gaussian semiring. Note that $x \oplus y = 1$ iff either $x=1, y=0$ or $x=0, y=1$. Also $x\odot y = 1$ iff $x=y=1$. Therefore the set $P = B(n,i)-\{1\}$ is a prime ideal of the semiring $B(n,i)$. On the other hand  $P$ is not subtractive, since if $a \not= 0,1$, then $a\oplus 1 \in P$ and $a\in P$ but $1\notin P$. So for $i>1$, the semiring $B(n,i)$ is not a weak Gaussian semiring. Note that if $i \leq 1$, then $B(n,i)$ is a subtractive semiring.

\item Let $T$ be a semiring with the following properties: $(i)$: $a+b=1$ implies either $a=1$ or $b=1$ for all $a,b \in T$, $(ii)$: $ab=1$ implies $a=b=1$ for all $a,b \in T$. We let $P = T[X]-\{1\}$. One can easily check that $P$ is a prime ideal of $T[X]$, while it is not subtractive, since $X,X+1 \in P$ but $1\notin P$. So $T[X]$ is not a weak Gaussian semiring. We obtain another good example if we set $T=\mathbb B$.

\item Let $\mathbb N_0$ denote the set $\{0,1,2, \ldots,k,k+1, \ldots\}$, i.e. the set of all nonnegative integers. We define addition and multiplication over $S = \mathbb N_0 \cup \{- \infty \}$ as $``\max"$ and $``+"$ respectively by considering that $ - \infty < n < n+1$ for all $n\in \mathbb N_0$ and $-\infty + s = -\infty$ for all $s\in S$. One can easily check that $(\mathbb N_0 \cup \{- \infty \}, \max, +, -\infty, 0)$ is a semiring known as the Arctic semiring.

    Let $I$ be a non-trivial ideal of the Arctic semiring $S$, i.e. $I \not= \{- \infty \}$ and $I \not= S$. Then there exists a positive integer $k\in \mathbb N$ such that $k\in I$. This implies that $1\in \sqrt I$ and finally $\sqrt I = \mathbb N \cup \{- \infty \}$. But $\sqrt I$ is not a subtractive ideal of $S$, since $\max\{k,0\} \in \sqrt I$ and $k \in \sqrt I$, while $0\notin \sqrt I$. So by Theorem \ref{WDMcriteria1}, the Arctic semiring $(\mathbb N_0 \cup \{- \infty \}, \max, +, -\infty, 0)$ is not a weak Gaussian semiring.
\end{enumerate}

\end{example}

What we have seen up until now and the matters related to content algebras in the papers \cite{Na}, \cite{OR} and \cite{R} inspire us to introduce content and weak content semialgebras and generalize some interesting results for them. For doing that, we need to be familiar with content semimodules, the ones that we will introduce in the next section.

\section{Content Semimodules}

We introduce the concept of content semimodules that is a generalization of the concept of content modules introduced in \cite{OR}. We also bring some routine generalizations of the theorems on content modules, though we do not go through them deeply. Actually it was possible to ask the reader to refer to the papers \cite{OR}, \cite{ES} and \cite{R} to see the process of configuring (weak) content algebras from content modules and model them to configure (weak) content semialgebras from content semimodules, but we thought it was a good idea to bring the sketch of the process. Note that the reader who is familiar with content modules, may skip this section without losing the flow. Before introducing content semimodules, first we prove the following proposition:

\begin{proposition}

\label{contentsemimodule}

Let $M$ be an $S$-semimodule. The following statements are equivalent:

\begin{enumerate}
\item There is a function $d: M \longrightarrow \Id(S)$ such that $d(x) \subseteq I$ iff $x\in IM$ for all $x\in M$,
\item For all $x\in M$ ($x\in c_M(x)M$), where $c_M(x) = \bigcap \{ I \colon I \text{~is an ideal of~} S \text{~and~} x \in IM \}$,
\item $\cap_{\alpha} (I_{\alpha} M) = (\cap_{\alpha} I_{\alpha}) M$, for any family of ideals $\{I_\alpha\}$ of $S$.
\end{enumerate}

Moreover if one of these three conditions is satisfied, then $c_M(x)$ is a finitely generated ideal of $S$.

\begin{proof}

$(1) \rightarrow (2)$: Let $M$ be an $S$-semimodule and there is a function $d: M \longrightarrow \Id(S)$ such that $d(x) \subseteq I$ iff $x\in IM$ for all $x\in M$. Suppose $x\in M$. Since $d(x) \subseteq d(x)$, we have $x \in d(x)M$. Now we set $c_M(x) = \cap \{I: I\in \Id(S), x\in IM \}$. Obviously $d(x) \subseteq c_M(x)$. But $x\in d(x)M$, so $c_M(x) \subseteq d(x)$ and finally $d(x) = \cap \{I: I\in \Id(S), x\in IM \}$.

$(2) \rightarrow (3)$: By considering $x\in c_M(x)M$, it is obvious that $\cap_{\alpha} (I_{\alpha} M) \subseteq (\cap_{\alpha} I_{\alpha}) M$, for any family of ideals $\{I_\alpha\}$ of $S$.

$(3) \rightarrow (1)$: Define $d=c_M(x)$.

Note that if $x\in c_M(x)M$, then $x = c_1 x_1 + c_2 x_2 + \cdots + c_n x_n$, where $c_i \in c_M(x)$ and $x_i \in M$ for all $1 \leq i \leq n$. This means that $x \in (c_1, c_2, \ldots , c_n)M$ and therefore $c_M(x) \subseteq (c_1, c_2, \ldots , c_n)$, but $c_i \in c_M(x)$ for all $1 \leq i \leq n$. Hence $c_M(x) = (c_1, c_2, \ldots , c_n)$ is a finitely generated ideal of $S$ and this is what we wanted to show.
\end{proof}

\end{proposition}

 One can easily see that if $S$ is a semiring and $\Lambda , \Delta$ are two index sets. Then for each $f \in S[X_\Lambda , {X_\Delta}^{-1}]$,  the ideal $c_{S[X_\Lambda , {X_\Delta}^{-1}]} (f)$ coincides with $c(f)$ the ideal generated by the coefficients of the polynomial $f$, since they both satisfy the condition 1 in Proposition \ref{contentsemimodule}. This motivates us to give the following definition that is similar to the definition of content modules in \cite[Definition 1.1]{OR}.

\begin{definition}

The $S$-semimodule $M$ is said to be a \emph{content} semimodule if $x\in c_M(x)M$ for all $x\in M$, where $c_M(x) = \bigcap \{ I \colon I \text{~is an ideal of~} S \text{~and~} x \in IM \}$.

\end{definition}

 The function $c_M: M \longrightarrow \Id(S)$ with $c_M(x) = \bigcap \{ I \colon I \text{~is an ideal of~} S \text{~and~} x \in IM \}$ is called the \emph{content} function. Note that when there is no fear of confusion, the subscript $M$ in $c_M(x)$ will be omitted. Now we give the following theorem similar to Theorem 1.3 in \cite{OR}. Since its proof is just a mimic of the proof of Theorem 1.3 in \cite{OR}, we omit it.

\begin{theorem}

Let $M$ be a content $S$-semimodule and $N$ be an $S$-subsemimodule of $M$. Then the following statements are equivalent:

\begin{enumerate}
\item $IM \cap N = IN$ for all ideals $I$ of $S$,
\item For all $x\in N$, $x\in c_M(x)N$,
\item $N$ is a content module and $c_M$ restricted to $N$ is $c_N$.
\end{enumerate}

\end{theorem}

The first corollary of the above theorem is the following assertion. The proof of this corollary is nothing but just the mimic of the proof of Corollary 1.4 in \cite{OR}:

\begin{corollary}
Let $\{M_i\}$ be a family of $S$-semimodules. Then the $S$-semimodule $\oplus_i M_i$ is a content $S$-semimodule iff each $M_i$ is a content $S$-semimodule.
\end{corollary}

This corollary implies the following corollary:

\begin{corollary}
If $S$ is a semiring and $\Lambda , \Delta$ are two index sets, then $S[X_\Lambda , {X_\Delta}^{-1}]$ is a content $S$-semimodule.
\end{corollary}

Now let for the moment $S$ be a Noetherian semiring, i.e. every ideal of $S$ is finitely generated, and let $\{M_i\}$ be a family of $S$-semimodules. Then one can easily check that $J \prod_i M_i = \prod_i JM_i$ for any ideal $J$ of $S$. By using this, the following assertion is obtained. The proof of this corollary is nothing but just the mimic of the Corollary 2.6 in \cite{ES}.

\begin{corollary}
Let $S$ be a Noetherian semiring and $\{M_i\}$ be a family of $S$-semimodules. Then the $S$-semimodule $\prod_i M_i$ is a content $S$-semimodule iff each $M_i$ is a content $S$-semimodule.
\end{corollary}

The following corollary is implied by the above corollary:

\begin{corollary}
If $S$ is a Noetherian semiring, then $S[[X_1, X_2, \ldots , X_n]]$ is a content $S$-semimodule.
\end{corollary}

To discuss more on content semimodules may cause us to go too far away from the main purpose of the paper, so we end this section here and pass to the next section to discuss content semialgebras.

\section{Content Semialgebras and Their Prime Ideals}

Now we go further to define content semialgebras that is a generalization of polynomial semirings. Then we will see the applications of Dedekind-Mertens content formula for semialgebras over semirings. Note that if $B$ is an $S$-semialgebra and a content $S$-semimodule, then according to the definition of content semimodules, for all $f\in B$, we have $f\in c(f)B$ and therefore if $f,g\in B$, we have $f\in c(f)B$ and $g\in c(g)B$ and this implies that $fg\in c(f)c(g)B$ and finally $c(fg) \subseteq c(f)c(g)$ and so for all $m\in \mathbb N_0$, we have $c(f)^m c(fg) \subseteq c(f)^{m+1} c(g)$. Since Dedekind-Mertens content formula, i.e $c(f)^{m+1} c(g) = c(f)^m c(fg)$ has some interesting applications in commutative algebra (cf. \cite{AG}, \cite{AK}, \cite{BG}, \cite{HH}, \cite{Na}, \cite{No}, \cite{OR}, \cite{R} and \cite{T}), we are motivated to define the concept of content semialgebras similar to content algebras introduced in \cite{OR}.

\begin{definition}

\label{defcontentsemialgebra}

Let $B$ be an $S$-semialgebra. We say that $B$ is a \emph{content} $S$-semialgebra if the homomorphism $\lambda$ is injective (i.e. we can consider $S$ as a subsemiring of $B$) and there exists a function $c: B \longrightarrow \Id(S)$ such that the following conditions hold:

\begin{enumerate}

\item $f\in IB$ iff $c(f) \subseteq I$ for all ideals $I$ of $S$;

\item $c(sf) = sc(f)$ for all $s\in S$ and $f\in B$ and $c(1)= S$;

\item (Dedekind-Mertens content formula) For all $f,g \in B$ there exists an $m\in \mathbb N_0$ such that $c(f)^{m+1} c(g) = c(f)^m c(fg)$.

\end{enumerate}

\end{definition}

The first example of a content $S$-semialgebra that may come to one's mind is that of a Laurent polynomial semiring over the subtractive semiring $S$ in any number of indeterminates, i.e. the $S$-semialgebra $S[X_\Lambda , {X_\Delta}^{-1}]$ in our terminology. We gather the basic properties of content semialgebras in the following proposition:

\begin{proposition}

\label{mccoyproperty}

Let $B$ be a content $S$-semialgebra. Then the following statements hold:

\begin{enumerate}

\item $c(f) = \cap \{I: I\in \Id(S), f\in IB \}$ and $f\in c(f)B$ for all $f\in B$;
\item $c(f)$ is a finitely generated ideal of $S$ for all $f \in B$;
\item $c(fg) \subseteq c(f)c(g)$ for all $f,g \in B$;
\item $c(fg) = S$ iff $c(f) = c(g) = S$ for all $f,g \in B$;
\item If $\textbf{p}$ is a prime ideal of $S$, then $\textbf{p}B$ is a prime ideal of $B$;
\item (McCoy's Property \cite{M}) If $fg = 0$ and $g \not= 0$, then there exists a nonzero $s\in S$ such that $sf = 0$.

\end{enumerate}

\begin{proof}
$(1)$ and $(2)$ are nothing but Proposition \ref{contentsemimodule} and $(3)$ and $(4)$ are straightforward. We prove assertions $(5)$ and $(6)$.

$(5)$: As we know $\textbf{p}$ is a prime ideal of $S$ iff $IJ \subseteq \textbf{p}$ implies either $I \subseteq \textbf{p}$ or $J \subseteq \textbf{p}$ for all ideals $I,J$ of $S$. First note that since $\textbf{p}$ is a prime ideal, $\textbf{p} \not= S$. We claim that $\textbf{p}B \not= B$. On the contrary if $\textbf{p}B = B$, then $1 \in \textbf{p}B$ and therefore $S = c(1) \subseteq \textbf{p}$, a contradiction. So $\textbf{p}B \not= B$. Now let $fg \in \textbf{p}B$. Therefore $c(fg) \subseteq \textbf{p}$ and from Dedekind-Mertens content formula in the definition of content $S$-semialgebras, there exists an $m \in \mathbb N_0$ such that $c(f)^{m+1} c(g) \subseteq \textbf{p}$. Obviously this causes either $c(f) \subseteq \textbf{p}$ or $c(g) \subseteq \textbf{p}$ and this means that either $f \in \textbf{p}B$ or $g \in \textbf{p}B$ and the trueness of the statement $(5)$ is showed.

$(6)$: Consider that if $fg=0$ and $g \not= 0$, then from Dedekind-Mertens content formula in the definition of content $S$-semialgebras, there exists an $m \in \mathbb N_0$ such that $c(f)^{m+1} c(g)=(0)$. Let $t\in \mathbb N_0$ be the smallest number such that $c(f)^{t+1} c(g)=(0)$. Therefore $c(f)^t c(g) \not= (0)$ and for all $s \in c(f)^t c(g)-\{0\}$, we have $sf=0$.
\end{proof}

\end{proposition}

Now we give a general theorem on minimal prime ideals in semialgebras. One of the results of this theorem is that in content semialgebras, minimal primes extend to minimal primes and, more precisely, there is actually a correspondence between the minimal primes of the semiring and their extensions in the semialgebra. Note that if $S$ and $B$ are two semirings and $\lambda: S \longrightarrow B$ is a semiring homomorphism and $P$ is an (a prime) ideal of $B$, then its contract $\lambda^{-1} (P)$, denoted by $P \cap S$, is also (a prime) an ideal of $S$.

\begin{theorem}

\label{Minimalprimes1}

 Let $B$ be an $S$-semialgebra with the following properties:

\begin{enumerate}
 \item For each prime ideal $\textbf{p}$ of $S$, the extended ideal $\textbf{p}B$ of $B$ is prime;
 \item For each prime ideal  $\textbf{p}$ of $S$, $\textbf{p}B \cap S = \textbf{p}$.
\end{enumerate}

Then the function $ \phi : \Min(S) \longrightarrow \Min(B)$ given by $\textbf{p} \longrightarrow \textbf{p}B$ is a bijection.

\begin{proof}
First we prove that if $\textbf{p}$ is a minimal prime ideal of $S$, then $\textbf{p}B$ is also a minimal prime ideal of $B$. Let $Q$ be a prime ideal of $B$ such that $Q \subseteq \textbf{p}B$. So $Q \cap S \subseteq \textbf{p}B \cap S = \textbf{p}$. Since $\textbf{p}$ is a minimal prime ideal of $S$, we have $Q \cap S =\textbf{p}$ and therefore $ Q = \textbf{p}B $. This means that $ \phi $ is a well-defined function. Obviously the second condition causes $ \phi $ to be one-to-one. The next step is to prove that $ \phi $ is onto. For showing this, consider $Q \in \Min(B)$, so $Q \cap S$ is a prime ideal of $S$ such that $(Q \cap S)B \subseteq Q$ and therefore $(Q \cap S)B = Q$. Our claim is that $(Q \cap S)$ is a minimal prime ideal of $S$. Suppose $\textbf{p}$ is a prime ideal of $S$ such that $ \textbf{p} \subseteq Q \cap S$, then $\textbf{p}B \subseteq Q$ and since $Q$ is a minimal prime ideal of $B$, $\textbf{p}B = Q = (Q \cap S)B$ and therefore $\textbf{p} = Q \cap S$.
\end{proof}

\end{theorem}

Let $S$ be a semiring. An element $s\in S$ is said to be nilpotent if $s^n = 0$ for some $n\in \mathbb N$. The set of all nilpotent elements of the semiring $S$ is called the \emph{lower nil radical} of $S$ and in this paper, we denote that by $\Nil (S)$. A semiring $S$ is said to be reduced if $\Nil(S) = (0)$. Note that by \cite[Proposition 7.28 (Krull's Theorem)]{G}, we have $\Nil(S) = \bigcap_{\textbf{p}\in \Min(S)} \textbf{p}$. Now we can easily prove the following:

\begin{corollary}
Let the $S$-semialgebra $B$ be a content $S$-semimodule such that the following statements hold:

\begin{enumerate}
 \item For each prime ideal $\textbf{p}$ of $S$, the extended ideal $\textbf{p}B$ of $B$ is prime;
 \item For each prime ideal  $\textbf{p}$ of $S$, $\textbf{p}B \cap S = \textbf{p}$.
\end{enumerate}

Then $\Nil(B) = \Nil(S)B$. Particularly $B$ is reduced iff $S$ is reduced.

\begin{proof}
By considering Corollary \ref{Minimalprimes1} and Proposition \ref{contentsemimodule}, we have: $\Nil(B) = \bigcap_{P\in \Min(B)} P = \bigcap_{\textbf{p}\in \Min(S)} \textbf{p}B = (\bigcap_{\textbf{p}\in \Min(S)} \textbf{p})B = \Nil(S)B$.
\end{proof}

\end{corollary}

\begin{corollary}

\label{Minimalprimes2}

 Let $B$ be a content $S$-semialgebra. Then the following statements hold:
 \begin{enumerate}
 \item The function $ \phi : \Min(S) \longrightarrow \Min(B)$ given by $\textbf{p} \longrightarrow \textbf{p}B$ is a bijection.
 \item $\Nil(B) = \Nil(S)B$. Particularly $B$ is reduced iff $S$ is reduced.
 \end{enumerate}

\end{corollary}

\begin{corollary}
 Let $S$ be a subtractive semiring and $\Lambda , \Delta$ be two index sets. Then the following statements hold:
 \begin{enumerate}
 \item The function $ \phi : \Min(S) \longrightarrow \Min(S[X_\Lambda , {X_\Delta}^{-1}])$ given by $\textbf{p} \longrightarrow \textbf{p}(S[X_\Lambda, {X_\Delta}^{-1}])$ is a bijection;
\item $\Nil(S[X_\Lambda , {X_\Delta}^{-1}]) = \Nil(S).S[X_\Lambda , {X_\Delta}^{-1}]$. Particularly $S[X_\Lambda , {X_\Delta}^{-1}]$ is reduced iff $S$ is reduced.
\end{enumerate}

\end{corollary}

\section{Weak Content Semialgebras}

The proof of Theorem \ref{WDMcriteria1} in section 4 shows us that the content formula $c(f)c(g) \subseteq \sqrt {c(fg)}$ for polynomial semirings in a single indeterminate is important. The concept of weak content algebras introduced in \cite{R} satisfies the same content formula and is shown to be an interesting generalization of formal power series rings. All of these inspire us to introduce weak content semialgebras. Later in this section we will give more interesting examples of weak content semialgebras. First we define weak content semialgebras:

\begin{definition}

\label{weakcontentsemialgebra}

Let $S$ be a semiring and $B$ be an $S$-semialgebra. We say $B$ is a weak content $S$-semialgebra if the following conditions holds:

\begin{enumerate}

\item $B$ is a content $S$-semimodule;

\item $c(f)c(g) \subseteq \sqrt {c(fg)}$, for all $f,g \in B$.

\end{enumerate}

\end{definition}

Note that if $B$ is a weak content $S$-semialgebra, then by condition (1) of Definition \ref{weakcontentsemialgebra}, $c(fg) \subseteq c(f)c(g)$ for all $f,g \in B$. Now we proceed to bring the semialgebra version of \cite[Theorem 1.2]{R}:

\begin{proposition}

\label{weakcontentformula1}

Let the $S$-semialgebra $B$ be a content $S$-semimodule. Then the following statements are equivalent:

\begin{enumerate}
\item $c(f)c(g) \subseteq \sqrt {c(fg)}$, for all $f,g \in B$,
\item For each prime ideal $\textbf{p}$ of $S$, either $\textbf{p}B$ is a prime ideal of $B$ or $\textbf{p}B=B$.
\end{enumerate}

\begin{proof}
$(1) \rightarrow (2)$: Let $\textbf{p}$ be a prime ideal of $R$ such that $\textbf{p}B \not= B$. If $fg\in \textbf{p}B$, then $c(fg) \subseteq \textbf{p}$ and so $c(f)c(g) \subseteq \textbf{p}$. Thus either $c(f) \subseteq \textbf{p}$ or $c(g) \subseteq \textbf{p}$ and this means either $f\in \textbf{p}B$ or $g\in \textbf{p}B$. So we have already showed that $\textbf{p}B$ is a prime ideal of $B$.

$(2) \rightarrow (1)$: Let $f,g \in B$ such that $c(fg) \subseteq \textbf{p}$, where $\textbf{p}$ is a prime ideal of $S$. Obviously $fg \in \textbf{p}B$. Then $(2)$ implies that either $f\in \textbf{p}B$ or $g\in \textbf{p}B$ and this means that either $c(f) \subseteq \textbf{p}$ or $c(g) \subseteq \textbf{p}$ and in any case $c(f)c(g) \subseteq \textbf{p}$. Consequently $c(f)c(g) \subseteq \bigcap_{\textbf{p}\in \Spec_{c(fg)}(S)} \textbf{p} = \sqrt {c(fg)}$ and this is what we wanted to show.
\end{proof}
\end{proposition}

\begin{corollary}

\label{weakcontentformula2}

Let $B$ be a content $S$-semialgebra. Then $c(fg) \subseteq c(f)c(g) \subseteq \sqrt {c(fg)}$, for all $f,g \in B$.

\begin{proof}
In content semialgebras, primes extend to primes (Proposition \ref{mccoyproperty}).
\end{proof}
\end{corollary}

 Corollary \ref{weakcontentformula2} shows that content semialgebras are weak content semialgebras, but the converse is not true. For example, if $R$ is a Noetherian ring, then $R[[X_1,X_2,\ldots,X_n]]$ is a weak content $R$-algebra and obviously a weak content $R$-semialgebra. In fact, Epstein and Shapiro, as the main theorem of their paper \cite[Theorem 2.6]{ESh}, prove that if $R$ is a Noetherian ring, then the Dedekind-Mertens content formula holds for $R[[X]]$, while for general commutative rings, there is no exponent for which the content formula holds in power series rings as they show in \cite[Example 4.1]{ESh}.

Let us remind that a celebrated theorem in semiring theory says that if $P$ is an ideal of a semiring $S$ and $X$ is an indeterminate over $S$, then $P[X]$ is a prime ideal of $S[X]$ iff $P$ is a subtractive prime ideal of $S$ (\cite[Theorem 2.6]{L}). By using Lemma \ref{supppoly}, we show that the same statement holds for the Laurent polynomial semirings in an arbitrary number of indeterminates.

\begin{theorem}

\label{primeextension}

Let $P$ be an ideal of a semiring $S$ and $\Lambda , \Delta$ be two index sets such that $\Lambda \cup \Delta \not= \emptyset$. Then $P[X_\Lambda , {X_\Delta}^{-1}]$ is a prime ideal of $S[X_\Lambda , {X_\Delta}^{-1}]$ iff $P$ is a subtractive prime ideal of $S$.

\begin{proof}
$(\rightarrow)$: Suppose that $P[X_\Lambda , {X_\Delta}^{-1}]$  is a prime ideal of $S[X_\Lambda , {X_\Delta}^{-1}]$ and $a,b\in S$ such that $ab\in P$. Then either $a\in P[X_\Lambda , {X_\Delta}^{-1}]$ or $b\in P[X_\Lambda , {X_\Delta}^{-1}]$. But $P[X_\Lambda , {X_\Delta}^{-1}] \cap S = P$ and so we have proved that $P$ is a prime ideal of $S$. Now suppose that $a,b\in S$ such that $a+b,a \in P$ and $X \in X_\Lambda \cup {X_\Delta}^{-1}$ is an indeterminate. Put $f=a+bX$ and $g=b+(a+b)X$. We have $fg= ab+(a^2+ab+b^2)X+(ab+b^2)X^2$ and so $fg\in P[X_\Lambda , {X_\Delta}^{-1}]$. But since $P[X_\Lambda , {X_\Delta}^{-1}]$ is prime, either $f\in P[X_\Lambda , {X_\Delta}^{-1}]$ or $g\in P[X_\Lambda , {X_\Delta}^{-1}]$ and in each case, $b\in P$ and so we have showed that $P$ is subtractive.

$(\leftarrow)$: Suppose that $P$ is a subtractive prime ideal of $S$. Obviously in order to show that $P[X_\Lambda , {X_\Delta}^{-1}]$ is a prime ideal of $S[X_\Lambda , {X_\Delta}^{-1}]$, we need to prove that if $X_1,X_2, \ldots, X_n \in X_\Lambda \cup {X_\Delta}^{-1}$, then $P[X_1,X_2, \ldots, X_n]$ is a prime ideal of $S[X_1,X_2, \ldots, X_n]$. Our proof is by induction on $n$. For $n=1$, the result follows from \cite[Theorem 2.6]{L}. If the result is true for $n=k$ and $f,g\in S[X_1, \ldots, X_k, X_{k+1}]$ such that $fg\in P[X_1, \ldots, X_k, X_{k+1}]$, then without loss of generality, we may assume that both $f$ and $g$ are nonzero polynomials and so Lemma \ref{supppoly} implies the existence of $f^*, g^* \in S[X_1, \ldots, X_k]$ such that $c(f) = c(f^*)$,  $c(g) = c(g^*)$ and $c(fg)=c(f^* g^*)$. Since $fg\in P[X_1, \ldots X_k, X_{k+1}]$, we have $c(fg) \subseteq P$ and so $c(f^* g^*) \subseteq P$. Thus $f^* g^* \in P[X_1,X_2, \ldots, X_k]$. But by induction's assumption, $P[X_1,X_2, \ldots, X_k]$ is a prime ideal of $S[X_1,X_2, \ldots, X_k]$, so either $f^* \in P[X_1,X_2, \cdots, X_k]$ or $g^* \in P[X_1,X_2, \ldots, X_k]$. This implies that either $c(f^*) \subseteq P$ or $c(g^*) \subseteq P$. Thus either $c(f) \subseteq P$ or $c(g) \subseteq P$. This means that either $f\in P[X_1, \ldots, X_k, X_{k+1}]$ or $g\in P[X_1, \ldots, X_k, X_{k+1}]$ and the proof is complete.
\end{proof}

\end{theorem}

The next theorem is a generalization of Theorem \ref{WDMcriteria1} in section 4 and another good family of weak content semialgebras:

\begin{theorem}

\label{WDMcriteria2}

Let $S$ be a semiring and $\Lambda , \Delta$ be two index sets such that $\Lambda \cup \Delta \not= \emptyset$. Then the following statements are equivalent:

\begin{enumerate}

\item $c(fg) \subseteq c(f)c(g) \subseteq \sqrt {c(fg)}$, for all $f,g \in S[X_\Lambda , {X_\Delta}^{-1}]$,
\item $\sqrt I$ is subtractive for each ideal $I$ of the semiring $S$,
\item Each prime ideal $\textbf{p}$ of $S$ is subtractive.

\end{enumerate}

\begin{proof}
$(1) \rightarrow (2)$: Suppose $I$ is an ideal of $S$ and $a,b \in S$ such that $a+b, a \in \sqrt I$. We need to show that $b\in \sqrt I$. Let $X \in X_\Lambda \cup {X_\Delta}^{-1}$ be an indeterminate and put $f=a+bX$ and $g=b+(a+b)X$. Then just like the proof of Theorem \ref{WDMcriteria1}, we have $b \in \sqrt I$.

$(2) \rightarrow (3)$ is obvious and by considering Theorem \ref{primeextension}, the proof of $(3) \rightarrow (1)$ is just the mimic of the proof of Theorem \ref{WDMcriteria1}.
\end{proof}

\end{theorem}

\begin{corollary}
Let $S$ be a subtractive semiring and $\Lambda , \Delta$ be two index sets. Then $c(fg) \subseteq c(f)c(g) \subseteq \sqrt {c(fg)}$ for all $f,g \in S[X_\Lambda , {X_\Delta}^{-1}]$. In particular, if $S$ is a reduced subtractive semiring, then $fg=0$ implies $\supp(f)\supp(g)=\{0\}$ for all $f,g \in S[X_\Lambda , {X_\Delta}^{-1}]$.
\end{corollary}

 Now let for the moment $S$ be a semiring and $X_1,X_2,\ldots,X_n$ be $n$ indeterminates over $S$. Obviously the set of all formal power series $S[[X_1,X_2,\ldots, X_n]]$ is an $S$-semialgebra. In the rest of this section, we discuss this important family of semialgebras and show that under suitable conditions, they are also good examples for weak content semialgebras.

 \begin{proposition}

 \label{formalpowerseries}

 If for any $f\in S[[X_1,X_2,\ldots, X_n]]$, we set $A_f$ to be the ideal generated by the coefficients of $f$, then $A_{fg} \subseteq A_fA_g$ for all $f,g \in S[[X_1,X_2,\ldots, X_n]]$. Moreover if $S$ is a Noetherian semiring, then the following statements hold:

\begin{enumerate}

\item\ If $I$ is an ideal of $S$, then $I \cdot S[[X_1,X_2,\ldots, X_n]] = I[[X_1,X_2,\ldots, X_n]]$;

\item The function $A_f : S[[X_1,X_2,\ldots, X_n]] \longrightarrow \Id(S)$ satisfies the following property: $A_f \subseteq I$ iff $f\in I \cdot S[[X_1,X_2,\ldots, X_n]]$ for each ideal $I$ of $S$;

\item $c(f) = A_f$ for any $f\in S[[X_1,X_2,\ldots, X_n]]$ and $S[[X_1,X_2,\ldots, X_n]]$ is a content $S$-semimodule.

\end{enumerate}

\begin{proof} The statement $A_{fg} \subseteq A_fA_g$ for all $f,g \in S[[X_1,X_2,\ldots, X_n]]$ is straightforward.

Now let $S$ be a Noetherian semiring. Then $(1)$ is straightforward. We prove $(2)$ and $(3)$. The proof of $(2)$ is as follows:

$(\rightarrow)$: Let $f\in S[[X_1,X_2,\ldots, X_n]]$ and $I$ be an ideal of $S$ such that $A_f \subseteq I$. Therefore $A_f[[X_1,X_2,\ldots, X_n]] \subseteq I[[X_1,X_2,\ldots, X_n]]$ and so $f\in I[[X_1,X_2,\ldots, X_n]]$. But since $S$ is a Noetherian ring, $I \cdot S[[X_1,X_2,\ldots, X_n]]= I[[X_1,X_2,\ldots, X_n]]$ and so $f \in I \cdot S[[X_1,X_2,\ldots, X_n]]$.

 $(\leftarrow)$: On the other hand if $f \in I \cdot S[[X_1,X_2,\ldots, X_n]]$, then $f \in I[[X_1,X_2,\ldots, X_n]]$ and so each coefficient of $f$ is an element of $I$ and finally $A_f \subseteq I$.

The proof of $(3)$ is as follows:

Since $S$ is a Noetherian semiring, by $(2)$ and Proposition \ref{contentsemimodule}, $S[[X_1,X_2,\ldots, X_n]]$ is a content $S$-semimodule and $c(f) = A_f$ for all $f\in S[[X_1,X_2,\ldots, X_n]]$.
 \end{proof}

 \end{proposition}

 Now we bring a celebrated theorem on the prime ideals of formal power series on semirings. Though the proof of the following lemma is rather similar to the proof of Theorem 2.6 in \cite{L}, we also bring its proof only for the sake of the completeness of the present paper.

 \begin{lemma}

 \label{primepowerseries}

 If $P$ is an ideal of a semiring $S$ and $X$ is an indeterminate over $S$, then $P[[X]]$ is a prime ideal of $S[[X]]$ iff $P$ is a subtractive prime ideal of $S$.

 \begin{proof}

 First we prove $(\leftarrow)$. Let $P$ be a subtractive prime ideal of $S$. We let $f,g \in S[[X]]$ such that $fg \in P[[X]]$. Therefore $A_{fg} \subseteq P$. If we set $f=a_0+a_1 X+\cdots+a_n X^n + \cdots$ and $g=b_0+b_1 X+\cdots+b_n X^n + \cdots$, then $a_0b_0 \in P$. Our claim is that $a_ib_j \in P$ for all $i,j \in \mathbb N_0$. First we prove $a_0b_j \in P$ by induction on $j$. Let $a_0b_k \in P$ for all $0\leq k \leq n$. Note that the coefficient of the monomial $X^{n+1}$ in $fg$ is also in $P$. This means that $a_0b_{n+1}+ a_1b_n + \cdots + a_{n+1}b_0 \in P$. Therefore $a_0(a_0b_{n+1}+ a_1b_n + \cdots + a_{n+1}b_0) \in P$.

 But $a_0(a_0b_{n+1}+ a_1b_n + \cdots + a_{n+1}b_0) = {a_0}^2 b_{n+1}+a_0(a_1b_n + \cdots + a_{n+1}b_0)$ and since $a_0b_k \in P$ for all $0\leq k \leq n$ and $P$ is subtractive, then ${a_0}^2 b_{n+1} \in P$ and this means that $(a_0b_{n+1})^2 \in P$. Hence $a_0b_{n+1} \in P$, since $P$ is a prime ideal of $S$. So for the moment we get that $a_0g \in P[[X]]$. But $P[[X]]$ is a subtractive ideal of $S[[X]]$, since $P$ is a subtractive ideal of $S$. So $(a_1X+a_2X^2+ \cdots + a_nX^n + \cdots)g \in P[[X]]$. We let $f_1=a_1+a_2X+ \cdots + a_{n+1}X^n + \cdots$, so $A_{f_1g} \subseteq P$. So the same process for $f_1$ shows us that $a_1b_j \in P$. Therefore if we set $f_i = a_i+a_{1+i}X+\cdots a_{n+i}X^n+\cdots$, then by induction on $i$, we get that $X^{i} f_i g \in P[[X]]$ and $A_{f_ig} \subseteq P$ and finally $a_ib_j \in P$ for all $i,j \in \mathbb N_0$. This means that $A_fA_g \subseteq P$ and since $P$ is a prime ideal of $S$, either $A_f \subseteq P$ or $A_g \subseteq P$ and at last either $f\in P[[X]]$ or $g\in P[[X]]$ and the proof of $(\leftarrow)$ is complete.

 Now we prove $(\rightarrow)$. Let $P$ be an ideal of $S$ such that $P[[X]]$ is a prime ideal of $S[[X]]$. We let $a,b\in S$ such that $ab\in P$. Therefore $ab\in P[[X]]$ and so either $a\in P[[X]]$ or $b\in P[[X]]$. This implies that either $a\in P$ or $b\in P$. In the next step we prove that $P$ is subtractive. We let $a,b\in S$ such that $a+b,a \in P$ and put $f=a+bX$ and $g=b+(a+b)X$. We have $fg= ab+(a^2+ab+b^2)X+(ab+b^2)X^2$ and therefore $fg\in P[[X]]$. But $P[[X]]$ is a prime ideal of $S[[X]]$. So either $f\in P[[X]]$ or $g\in P[[X]]$ and in any case $b\in P$ and the proof is complete.
 \end{proof}

 \end{lemma}

 \begin{corollary}

 \label{primepowerseries2}

If $P$ is an ideal of a semiring $S$ and $X_1,X_2,\ldots,X_n$ are $n$ indeterminate over $S$, then $P[[X_1,X_2,\ldots,X_n]]$ is a prime ideal of $S[[X_1,X_2,\ldots,X_n]]$ iff $P$ is a subtractive prime ideal of $S$.

\begin{proof}
First we prove $(\leftarrow)$ by induction on the number of indeterminates $n$. The case $n=1$ is nothing but Lemma \ref{primepowerseries}. Now assume the statement is true for $n=k$. We suppose $P$ is a subtractive ideal of $S$. By induction's hypothesis, $P[[X_1,X_2,\ldots,X_k]]$ is a prime ideal of $S[[X_1,X_2,\ldots,X_k]]$. Obviously $P[[X_1,X_2,\ldots,X_k]]$ is subtractive as well. We let $X_{k+1}$ be an indeterminate over $S[[X_1,X_2,\ldots,X_k]]$. Therefore by Lemma \ref{primepowerseries}, $P[[X_1,X_2,\ldots,X_k]][[X_{k+1}]]$ is a prime ideal of $S[[X_1,X_2,\ldots,X_k]][[X_{k+1}]]$ and this is what we wanted to show. The proof of $(\rightarrow)$ is nothing but just the mimic of $(\rightarrow)$ in proof of Lemma \ref{primepowerseries}.
\end{proof}

\end{corollary}

\begin{theorem}

\label{WDMcriteria3}

Let $S$ be a semiring. Then the following statements are equivalent:

\begin{enumerate}

\item $A_{fg} \subseteq A_fA_g \subseteq \sqrt {A_{fg}}$ for all $f,g\in S[[X_1,X_2,\ldots,X_n]]$,
\item $\sqrt I$ is subtractive for each ideal $I$ of the semiring $S$,
\item Each prime ideal $P$ of $S$ is subtractive.

\end{enumerate}

\begin{proof}
By considering Corollary \ref{primepowerseries2}, the proof is nothing but only the mimic of the proof of Theorem \ref{WDMcriteria1} by substituting $A_f$ for $c(f)$.
\end{proof}

\end{theorem}

Theorem \ref{WDMcriteria3} and Proposition \ref{formalpowerseries} give us the following corollary:

\begin{corollary}

\label{WDMcriteria4}

Let $S$ be a Noetherian semiring. Then the following statements are equivalent:

\begin{enumerate}

\item $S[[X_1,X_2,\ldots,X_n]]$ is a weak content $S$-semialgebra,
\item $\sqrt I$ is subtractive for each ideal $I$ of the semiring $S$,
\item Each prime ideal $P$ of $S$ is subtractive.

\end{enumerate}
\end{corollary}

By Theorem \ref{WDMcriteria1}, Theorem \ref{Minimalprimes1}, Proposition \ref{formalpowerseries} and Corollary \ref{primepowerseries2}, we get the following corollary:

\begin{corollary}

Let $S$ be a Noetherian and a weak Gaussian semiring. Then the following statements hold:
 \begin{enumerate}
 \item $ \phi : \Min(S) \longrightarrow \Min(S[[X_1,X_2,\ldots,X_n]])$ given by $\textbf{p} \longrightarrow \textbf{p}(S[[X_1,X_2,\ldots,X_n]])$ is a bijection;
\item $\Nil(S[[X_1,X_2,\ldots,X_n]]) = \Nil(S) \cdot S[[X_1,X_2,\ldots,X_n]]$. Especially $S[[X_1,X_2,\ldots,X_n]]$ is reduced iff $S$ is reduced.
\end{enumerate}

\end{corollary}

\section{Content Semialgebras over Semirings Having Few Zero-divisors}

For a semiring $S$, by $Z(S)$, we mean the set of zero-divisors of $S$ and by an associated prime ideal $\textbf{p}$ of the semiring $S$, we mean a prime ideal of $S$ such that $\textbf{p} = \Ann(a) $ for some $a\in S$, where $\Ann(a) = \{s \in S: as=0 \}$. We denote the set of all associated prime ideals of the semiring $S$ by $\Ass(S)$.

One of the most important theorems in commutative ring theory, known as \emph{Prime Avoidance Theorem}, says that if $I$ is an ideal of a ring $R$ and $P_i$ ($1 \leq i \leq n$) are prime ideals of $R$ and $I \subseteq \cup_{i=1}^n P_i$, then $I \subseteq P_i$ for some $i$. A similar assertion can be said for semirings: If $I$ is an ideal of a semiring $S$ and $P_i$ ($1 \leq i \leq n$) are subtractive prime ideals of $S$ and $I \subseteq \cup_{i=1}^n P_i$, then $I \subseteq P_i$ for some $i$. This version of Prime Avoidance Theorem for semirings, enjoys exactly the same proof of the ring version of Prime Avoidance Theorem mentioned in \cite[Theorem 81]{K}. Note that if $S$ is a semiring and $x\in S$, then $\Ann(x) = \{s \in S: sx = 0\}$ is a subtractive ideal of $S$, because if $r+s \in \Ann(x)$ and $r\in \Ann(x)$ for $r,s \in S$, then $(r+s)x = 0$ and $rx=0$ and therefore $sx=0$. This means that $s \in \Ann(x)$ and $\Ann(x)$ is a subtractive ideal of $S$. From all we said, we deduce that if $I \subseteq \cup_{i=1}^n P_i$, where $P_i \in \Ass(S)$, then $I\subseteq P_i = \Ann(x_i)$ for some $i$. Especially $I \cdot x_i =0$, which means that $I$ has a nonzero annihilator.

In \cite{Dav}, it has been defined that a ring $R$ has \textit{few zero-divisors}, if $Z(R)$ is a finite union of prime ideals. We present the following definition to prove some other theorems related to content semialgebras.

 \begin{definition}
  A semiring $S$ has \textit{very few zero-divisors}, if $Z(S)$ is a finite union of prime ideals in $\Ass(S)$.
 \end{definition}

 \begin{theorem}

 \label{noetherian}

 Every Noetherian semiring has very few zero-divisors.

 \begin{proof}

 Let $S$ be a Noetherian semiring. We prove in two steps that the semiring $S$ has very few zero-divisors.

 Step 1: The set of zero-divisors $Z(S)$ of $S$ is a set-theoretic union of maximal primes of $Z(S)$.

 Obviously $Z(S) = \cup_{s \not = 0} \Ann(s)$. Since $S$ is a Noetherian semiring, by \cite[Theorem 1.35]{L}, any $\Ann(s)$ for $s \not= 0$ is contained in a maximal one. Let $P=\Ann(u)$ be maximal among the ideals $\Ann(s)$, where $s \not = 0$ and suppose $ab \in P$ but $a \notin P$. This means that $au \not= 0$ and $P \subseteq \Ann(au)$ and therefore by maximality of $P$, we have $P=\Ann(au)$. Now since $b$ annihilates $au$, we have $b\in P$.

 Step 2: There exist only finitely many maximal primes of $Z(S)$.

Let $\{P_i = \Ann(a_i)\}_i$ be the family of maximal primes of $Z(S)$. Suppose $A$ is the ideal generated by the elements $\{a_i\}_i$, where $P_i = \Ann(a_i)$. Since $S$ is Noetherian, $A$ is generated by a finite number of elements in $\{a_i\}_i$, say $a_1, a_2, \ldots, a_n$. If any further $a$'s exist, can be written as a linear combination of $a_1, a_2, \ldots, a_n$, i.e. $a_{n+1} = t_1 a_1 + t_2 a_2 + \cdots + t_n a_n$ ($t_i \in S$). This implies that $P_1 \cap P_2 \cap \cdots \cap P_n \subseteq P_{n+1}$ and therefore $P_k \subseteq P_{n+1}$ for some $1 \leq k \leq n$, contradicting the maximality of $P_k$. Hence there are no further $a_i$'s or $P_i$'s and the proof is complete (cf. \cite[Theorem 6]{K}).
 \end{proof}

 \end{theorem}

\begin{theorem}
\label{veryfewzerodivisor}
 Let $B$ be a content $S$-semialgebra. Then the following statements hold:

 \begin{enumerate}

 \item If $S$ has very few zero-divisors, then $B$ has very few zero-divisors as well;
 \item If $S$ is a weak Gaussian semiring and $B$ has very few zero-divisors, then $S$ has very few zero-divisors.

 \end{enumerate}

\end{theorem}

\begin{proof}
 $(1)$: Let $Z(S) = \textbf{p}_1\cup \textbf{p}_2\cup \cdots \cup \textbf{p}_n$, where $\textbf{p}_i \in \Ass(S)$ for all $1 \leq i \leq n$. We show that $Z(B) = \textbf{p}_1B\cup \textbf{p}_2B\cup \cdots \cup \textbf{p}_nB$. Let $g \in Z(B)$, so there exists an $s \in S- \lbrace 0 \rbrace $ such that $sg = 0$ and so $sc(g) = (0)$. Therefore $c(g) \subseteq Z(S)$ and by considering that each ideal in $\Ass(S)$ is subtractive and according to the Prime Avoidance Theorem for semirings, we have $c(g) \subseteq \textbf{p}_i$, for some $1 \leq i \leq n$ and therefore $g \in \textbf{p}_iB$. Now let $g \in \textbf{p}_1B\cup \textbf{p}_2B\cup \cdots \cup \textbf{p}_nB$, so there exists an $i$ such that $g \in \textbf{p}_iB$ and therefore $c(g) \subseteq \textbf{p}_i$ and $c(g)$ has a nonzero annihilator and this means that g is a zero-divisor of $B$. Note that $\textbf{p}_iB \in \Ass(B)$, for all $1 \leq i \leq n$.

$(2)$: Let $Z(B)= \cup _{i=1}^n Q_i$, where $Q_i \in \Ass(B)$ for all $1\leq i \leq n$. Therefore $Z(S) = \cup _{i=1}^n (Q_i \cap S)$. Without loss of generality, we can assume that $Q_i \cap S \nsubseteq Q_j \cap S$ for all $i \neq j$. Now we prove that $Q_i \cap S\in \Ass(S)$ for all $1 \leq i \leq n$. Consider $f\in B$ such that $Q_i = \Ann (f)$ and $c(f)=(s_1,s_2,\ldots,s_m)$. It is easy to see that $Q_i \cap S = \Ann (c(f)) \subseteq \Ann(s_1) \subseteq Z(S)$ and since every prime ideal of $S$ is subtractive, by Prime Avoidance Theorem for semirings, $Q_i \cap S =\Ann(s_1)$.
\end{proof}

\begin{corollary}

Let $S$ be a subtractive semiring and $\Lambda , \Delta$ be two index sets. Then $S$ has very few zero-divisors iff $S[X_\Lambda , {X_\Delta}^{-1}]$ has very few zero-divisors.

\end{corollary}

There are non-Noetherian semirings that have very few zero-divisors. For example, let for the moment, $S$ be a subtractive semiring that it has very few zero-divisors and $X_1, X_2 , \cdots$ be infinitely many indeterminates. Then according to Theorem \ref{veryfewzerodivisor}, the semiring $S[X_1, X_2 , \cdots]$ has very few zero-divisors while it is not Noetherian. On the other hand, there are semirings (actually rings) that they have few zero-divisors but not very few zero-divisors \cite[Remark 10]{Na}.

\section{Content Semialgebras over Semirings Having Property (A)}

Now we give the following definition that is similar to the definition of Property (A) for rings in [HK] and prove some other results for content semialgebras.

\begin{definition}
  A semiring $S$ has \textit{Property (A)}, if each finitely generated ideal $I \subseteq Z(S)$ has a nonzero annihilator.
 \end{definition}

Let $S$ be a semiring. If $S$ has very few zero-divisors, then $S$ has Property (A), but there are some non-Noetherian semirings which do not have Property (A) [K, Exercise 7, p. 63]. The class of non-Noetherian semirings having Property (A) is quite large [H, p. 2].

\begin{theorem}
 Let $B$ be a content $S$-semialgebra such that the content function $c: B \longrightarrow \FId(S)$ is onto, where by $\FId(S)$, we mean the set of finitely generated ideals of $S$. Then the following statements are equivalent:

\begin{enumerate}
 \item $S$ has Property (A),
 \item For all $f \in B$, $f$ is a regular element of $B$ iff $c(f)$ is a regular ideal of $S$.
\end{enumerate}

\begin{proof}
 $(1) \rightarrow (2)$: Let $S$ have Property (A). If $f \in B$ is regular, then for all nonzero $s \in S$, $sf \not= 0$ and so for all nonzero $s \in S$, $sc(f) \not= (0)$, i.e. $\Ann(c(f)) = (0)$ and according to the definition of Property (A), $c(f) \not\subseteq Z(S)$. This means that $c(f)$ is a regular ideal of $S$. Now let $c(f)$ be a regular ideal of $S$, so $c(f) \not\subseteq Z(S)$ and therefore $\Ann(c(f)) = (0)$. This means that for all nonzero $s \in S$, $sc(f) \not= (0)$, hence for all nonzero $s \in S$, $sf \not= 0$. Since $B$ is a content $S$-semialgebra, $f$ is not a zero-divisor of $B$.

$(2) \rightarrow (1)$: Let $I$ be a finitely generated ideal of $S$ such that $I \subseteq Z(S)$. Since the content function $c: B \longrightarrow \FId(S)$ is onto, there exists an $f \in B$ such that $c(f) = I$. But $c(f)$ is not a regular ideal of $S$, therefore according to our assumption, $f$ is not a regular element of $B$. Since $B$ is a content $S$-semialgebra, there exists a nonzero $s \in S$ such that $sf = 0$ and this means that $sI = (0)$, i.e. $I$ has a nonzero annihilator.
\end{proof}

\begin{remark}
 In the above theorem the surjectivity condition for the content function $c$ is necessary, because obviously $S$ is a content $S$-semialgebra and the condition (2) is satisfied, while one can choose the semiring $S$ such that it does not have Property (A) [K, Exercise 7, p. 63].
\end{remark}

\end{theorem}

An element $r$ of a ring $R$ is said to be \textit{prime to an ideal} $I$ of $R$ if $I : (r) =I$, where by $I : (r)$, we mean the set of all elements $d$ of $R$ such that $dr \in I$ \cite[p. 223]{ZS}. Let $I$ be an ideal of $R$. We denote the set of all elements of $R$ that are not prime to $I$ by $S(I)$. It is obvious that $r\in S(I)$ iff $r+I$ is a zero-divisor of the quotient ring $R/I$. The ideal $I$ is said to be primal if $S(I)$ forms an ideal and in such a case, $S(I)$ is a prime ideal of $R$. A ring $R$ is said to be \textit{primal}, if the zero ideal of $R$ is primal \cite{Dau}. It is obvious that $R$ is primal iff $Z(R)$ is an ideal of $R$. This motivates us to define primal semirings.

\begin{definition}
A semiring $S$ is said to be primal if $Z(S)$ is an ideal of $S$.
\end{definition}

\begin{theorem}

\label{PrimalSemiring}

 Let $B$ be a content $S$-semialgebra. Then the following statements hold:

\begin{enumerate}
 \item If $B$ is primal, then $S$ is primal and $Z(B) = Z(S)B$;
 \item If $S$ is primal and has Property (A), then $B$ is primal, has Property (A) and $Z(B) = Z(S)B$;
 \item If the content function $c: B \longrightarrow \FId(S)$ is onto, then $B$ is primal iff $S$ is primal and has Property (A).
\end{enumerate}

\end{theorem}

\begin{proof}
$(1)$: Assume that $Z(B)$ is an ideal of $B$. We show that $Z(S)$ is an ideal of $S$. For doing that, it is enough to show that if $a,b \in Z(S)$, then $a+b \in Z(S)$. Let $a,b \in Z(S)$. Since $Z(S) \subseteq Z(B)$ and $Z(B)$ is an ideal of $B$, we have $a+b \in Z(B)$. This means that there exists a nonzero $g \in B$ such that $(a+b)g = 0$. Since $g\neq0$, we can choose $0 \neq d \in c(g)$ and we have $(a+b)d = 0$. Now it is easy to check that $Z(B) = Z(S)B$.

$(2)$: Let $S$ have Property (A) and $Z(S)$ be an ideal of $S$. We show that $Z(B) = Z(S)B$. Let $f \in Z(B)$, then according to McCoy's property for content semialgebras (Proposition \ref{mccoyproperty}), there exists a nonzero $s \in S$ such that $f.s = 0$. Therefore we have $c(f) \subseteq Z(S)$ and since $Z(S)$ is an ideal of $S$, $f \in Z(S)B$. Now let $f \in Z(S)B$, then $c(f) \subseteq Z(S)$. Since $S$ has Property (A), $c(f)$ has a nonzero annihilator and this means that $f$ is a zero-divisor in $B$. So we have already showed that $Z(B)$ is an ideal of $B$ and therefore $B$ is primal. Finally we prove that $B$ has Property (A). Assume that $J=(f_1,f_2, \ldots, f_n) \subseteq Z(B)$. Therefore $c(f_1), c(f_2),\ldots, c(f_n) \subseteq Z(S)$. But $Z(S)$ is an ideal of $S$ and $c(f_i)$ is a finitely generated ideal of $S$ for any $1\leq i \leq n$, so $I=c(f_1)+c(f_2)+\cdots + c(f_n) \subseteq Z(S)$ is a finitely generated ideal of $S$ and there exists a nonzero $s \in S$ such that $sI=0$. This causes $sJ=0$ and $J$ has a nonzero annihilator in $B$.

$(3)$: We just need to prove that if $B$ is primal, then $S$ has Property (A). For doing that, let $I \subseteq Z(S)$ be a finitely generated ideal of $S$. Since the content function is onto, there exists an $f\in B$ such that $I = c(f)$. Since $c(f) \subseteq Z(S)$, $f\in Z(B)$. According to McCoy's property for content semialgebras, we have $f \cdot s =0$ for some nonzero $s \in S$ and this means $I=c(f)$ has a nonzero annihilator and the proof is complete.
\end{proof}

\begin{corollary}
 Let $S$ be a subtractive semiring and $\Lambda, \Delta$ be two index sets such that either $\Lambda$ or $\Delta$ is nonempty. Then $S[X_\Lambda , {X_\Delta}^{-1}]$ is primal iff $S$ is primal and has Property (A).
\end{corollary}

Let $S$ be a weak Gaussian semiring that has few zero-divisors. One can consider $Z(S)= \cup _{i=1}^n \textbf{p}_i$ such that $\textbf{p}_i \nsubseteq \cup _{j=1 \wedge j \neq i}^n \textbf{p}_j$ for all $ 1\leq i \leq n$. Obviously we have $\textbf{p}_i \nsubseteq \textbf{p}_j$ for all $i \neq j$. Also, by using Prime Avoidance Theorem for semirings, it is easy to check that, if $Z(S)= \cup _{i=1}^n \textbf{p}_i$ and $Z(S)= \cup _{k=1}^m \textbf{q}_k$ such that $\textbf{p}_i \nsubseteq \textbf{p}_j$ for all $i \neq j$ and $\textbf{q}_k \nsubseteq \textbf{q}_l$ for all $k \neq l$, then $m=n$ and $\{\textbf{p}_1,\ldots,\textbf{p}_n\}=\{\textbf{q}_1,\ldots,\textbf{q}_n\}$, i.e. these prime ideals are uniquely determined. This is the base for the following definition:

\begin{definition}
 A weak Gaussian semiring $S$ is said to have \textit{few zero-divisors of degree} $n$, if $S$ has few zero-divisors and $n$ is the number of maximal primes of $Z(S)$. In such a case, we write $\zd(S) = n$.
\end{definition}

Note that if $Z(S)$ is an ideal of $S$, then it is prime. Hence $S$ is primal iff $S$ has few zero-divisors of degree one, i.e. $\zd(S)=1$. Now we generalize Theorem \ref{PrimalSemiring} as follows:

\begin{theorem}
 Let $B$ be a content $S$-semialgebra. Then the following statements hold for all natural numbers $n$:

\begin{enumerate}
 \item If $\zd(B) = n$, then $\zd(S) \leq n$;
 \item If $S$ is subtractive and has Property (A) and $\zd(S) = n$, then $\zd(B) = n$;
 \item If $S$ is subtractive and the content function $c: B \longrightarrow \FId(S)$ is onto, then $\zd(B) = n$ iff $\zd(S) = n$ and $S$ has Property (A).
\end{enumerate}

\end{theorem}

\begin{proof}
 $(1)$: Let $Z(B) = \bigcup _{i=1}^n Q_i$. We prove that $Z(S) = \bigcup _{i=1}^n (Q_i \cap S)$. In order to do that let $s\in Z(S)$. Since $Z(S) \subseteq Z(B)$, there exists an $i$ such that $s\in Q_i$ and therefore $s\in Q_i \cap S$. Now let $s\in Q_i \cap S$ for some $i$, then $s\in Z(B)$, and this means that there exists a nonzero $g\in B$ such that $sg = 0$ and at last $sc(g) = 0$. Choose a nonzero $d\in c(g)$ and we have $sd=0$.

$(2)$: Since $S$ is a subtractive semiring, similar to the proof of Theorem \ref{veryfewzerodivisor}, if $Z(S)= \bigcup _{i=1}^n \textbf{p}_i$, then $Z(B)= \bigcup _{i=1}^n \textbf{p}_iB$. Also it is obvious that $\textbf{p}_iB \subseteq \textbf{p}_jB$ iff $\textbf{p}_i \subseteq \textbf{p}_j$ for all $1 \leq i,j \leq n$. These two imply that $\zd(B) = n$.

$3$: $(\leftarrow)$ is nothing but $(2)$. For proving $(\rightarrow)$, consider that by $(1)$, we have $\zd(S) \leq n$. Now we prove that semiring $S$ has Property (A). Let $I \subseteq Z(S)$ be a finite ideal of $S$. Choose $f\in B$ such that $I=c(f)$. So $c(f) \subseteq Z(S)$ and by Prime Avoidance Theorem and $(1)$, there exists an $1 \leq i \leq n$ such that $c(f) \subseteq Q_i \cap S$. Therefore $f \in (Q_i \cap S)B$. But $(Q_i \cap S)B \subseteq Q_i$. So $f \in Z(B)$ and according to McCoy's property for content semialgebras, there exists a nonzero $s\in S$ such that $f.s=0$. This means that $I.s=0$ and $I$ has a nonzero annihilator. Now by $(2)$, we have $\zd(S) = n$.
\end{proof}

\begin{corollary}
 Let $S$ be a subtractive semiring and $\Lambda, \Delta$ be two index sets such that either $\Lambda$ or $\Delta$ is nonempty. Then for all natural numbers $n$, $\zd(S[X_\Lambda , {X_\Delta}^{-1}]) = n$ iff $\zd(S)=n$ and $S$ has Property (A).
\end{corollary}

There are many rings having few zero-divisors of degree $n$. For example if $R$ is a primal ring, then $\zd(\prod_{i=1}^n R)=n$ (\cite[Remark 17]{Na}). In the following, we give some examples of semirings that are not rings but have few zero-divisors of degree $n$.

\begin{remark}

Let $S$ be a semiring. Then the following statements hold:

\begin{enumerate}

\item If $S = S_1\times S_2\times \dots \times S_n$ such that $S_i$ is a semiring having few zero-divisors of degree $k_i$ for all $1\leq i \leq n$, then $\zd(S) = \zd(S_1)+ \zd(S_2)+ \dots + \zd(S_n)$.
\item If we set $S=\prod_{i=1}^n T$, where $T=\{0,u,1\}$ is the semiring in Proposition \ref{WeakGaussian2}, then $S$ is subtractive, has Property (A) and $\zd(S) = \zd(S[X]) = n$.

\end{enumerate}

\end{remark}

\section*{Acknowledgment} The author is supported in part by University of Tehran and wishes to thank Prof. Dara Moazzami for his support and encouragement.

\end{document}